\documentclass[11pt,reqno,letterpaper]{amsart}
\usepackage[utf8]{inputenc}
\usepackage[T1]{fontenc}
\usepackage{color}
\usepackage[colorlinks=false, allcolors=blue,backref=page]{hyperref}
\usepackage{amsmath, amssymb, amsthm}
\usepackage{mathrsfs}
\usepackage{mathtools}
\usepackage[noabbrev,capitalize,nameinlink]{cleveref}
\usepackage[noadjust]{cite}
\usepackage{graphics}
\usepackage{pifont}
\usepackage{tikz}
\usepackage{bbm}
\usepackage[T1]{fontenc}
\DeclareMathOperator*{\argmax}{argmax}

\usetikzlibrary{arrows}

\usepackage{environ}
\usepackage{framed}
\usepackage{url}
\usepackage[linesnumbered,ruled,vlined]{algorithm2e}
\usepackage[noend]{algpseudocode}
\usepackage[labelfont=bf]{caption}
\usepackage{cite}
\usepackage{framed}
\usepackage[framemethod=tikz]{mdframed}
\usepackage{appendix}
\usepackage{graphicx}
\usepackage[textsize=tiny]{todonotes}
\usepackage{tcolorbox}
\usepackage{enumerate}
\allowdisplaybreaks[1]
\usepackage{enumerate}
\usepackage{stmaryrd}
\usepackage[margin=1in]{geometry}
\usepackage[normalem]{ulem}
\usepackage{pgfplots}
\pgfplotsset{compat=1.3}
\usepackage{mathrsfs}
\usetikzlibrary{arrows}
\newcommand{\A}{\mc{A}^n}
\newcommand{\B}{\mc{B}^n}
\newcommand{\C}{\mc{C}^n}
\newcommand{\D}{\mc{D}^n}
\newcommand{\E}{\mc{E}^n}

\newcommand{\T}{\Omega}
\newcommand{\floor}[1]{\left\lfloor #1 \right\rfloor}
\newcommand{\ceil}[1]{\left\lceil #1 \right\rceil}
\newcommand{\paren}[1]{\left( #1 \right)}
\newcommand{\sqb}[1]{\left[ #1 \right]}
\newcommand{\set}[1]{\left\{ #1 \right\}}
\usepackage[shortlabels]{enumitem}
\crefformat{enumi}{#2#1#3}
\crefrangeformat{enumi}{#3#1#4 to~#5#2#6}
\crefmultiformat{enumi}{#2#1#3}
{ and~#2#1#3}{, #2#1#3}{ and~#2#1#3}
\DeclareSymbolFont{symbolsC}{U}{pxsyc}{m}{n}
\SetSymbolFont{symbolsC}{bold}{U}{pxsyc}{bx}{n}
\DeclareFontSubstitution{U}{pxsyc}{m}{n}
\DeclareMathSymbol{\medcircle}{\mathbin}{symbolsC}{7}
\crefname{equation}{}{} 
\AtBeginEnvironment{appendices}{\crefalias{section}{appendix}} 
\usepackage[color,final]{showkeys} 

\colorlet{refkey}{orange!20}
\colorlet{labelkey}{blue!30}

\numberwithin{equation}{section}
\newtheorem{theorem}{Theorem}[section]
\newtheorem{proposition}[theorem]{Proposition}
\newtheorem{lemma}[theorem]{Lemma}

\newtheorem{corollary}[theorem]{Corollary}
\newtheorem{conjecture}[theorem]{Conjecture}
\newtheorem*{question*}{Question}
\newtheorem{fact}[theorem]{Fact}

\theoremstyle{definition}
\newtheorem{definition}[theorem]{Definition}

\newtheorem*{definition*}{Definition}

\theoremstyle{remark}
\newtheorem*{remark}{Remark}

\newcommand{\mb}{\mathbb}
\newcommand{\mbf}{\mathbf}

\newcommand{\mc}{\mathcal}

\newcommand{\ol}{\overline}
\newcommand{\on}{\operatorname}

\newcommand{\fl}[1]{\left\lfloor#1\right\rfloor}

\newcommand{\U}{\on{U}}
\newcommand{\dtv}{d_{\on{TV}}}
\newcommand{\vol}[1]{\on{vol}\left( #1 \right)}

\let\originalleft\left
\let\originalright\right
\renewcommand{\left}{\mathopen{}\mathclose\bgroup\originalleft}
\renewcommand{\right}{\aftergroup\egroup\originalright}

\allowdisplaybreaks

\newif\ifpublic

\publictrue

\ifpublic

\newcommand{\ignore}[1]{}

\else

\fi
\title{Size of the Largest Sum-Free Subset of $[n]^3$ and $[n]^4$}

\begin{document}

\author{Saba Lepsveridze}
\address{Massachusetts Institute of Technology, Cambridge, MA, USA}
\email{sabal@mit.edu}

\author{Yihang Sun}
\address{Stanford University, Stanford, CA, USA}
\email{kimisun@stanford.edu}

\begin{abstract}
We determine the density of the largest sum-free subset of the lattice cube  $\{1, 2, \dots, n\}^d$ for $d = 3$ and $d = 4$. This solves a conjecture of Cameron and Aydinian in dimensions $3$ and $4$.  
\end{abstract}

\maketitle

\section{Introduction}\label{sec:intro}

\subsection{History}\label{sec:history}

A subset $S$ of an abelian group is \textit{sum-free} if no $x, y, z \in S$ satisfies $x + y =z$. Studying the maximum size of sum-free sets in different settings is a fundamental question in additive combinatorics. The primary aim of this paper is to determine the asymptotic density $c_d$ of the largest sum-free subset $\set{1,2, \ldots, n}^d$. For notational convenience, define   $[n] \triangleq  \set{0,1,\ldots, n-1}$. Let $S_{d,n}$ be the largest sum-free subset of $[n]^d$, then equivalently the question of interest is the value
\begin{equation}\label{eq:def-cd}
	c_d \triangleq  \limsup_{n \to \infty}\frac{\left|S_{d,n}\right|}{n^d}.
\end{equation}

For $d=1$, it is a folklore result that $c_1=1/2$. Determining the value of $c_d$ for $d\ge 2$, on the other hand, is nontrivial. This question was posed by Aydinian \cite[Problem~10]{Cam05} and Cameron \cite{Cam02}, and it also features as an open problem in the list compiled by Green \cite[Problem~6]{GreOp}. Elsholtz and Rackham \cite{ER17} recently resolved this question for the case where $d=2$. However, the problem remained open for dimensions $d \geq 3$. 

In this paper, we develop a general linear programming relaxation argument for arbitrary dimensions, reduce the problem to finding a dual weight function, and construct the weight function for $d = 3$ and $d = 4$, thereby answering the question completely for these dimensions.

Let $\mbf{1}$ denote the all-ones vector in $\mb{R}^d$ and define
 \begin{equation}
 \begin{aligned}
c_d^* &\triangleq  \max_{u\in [0, d]} \on{vol}_d\left(\{v\in [0, 1]^d: \mbf{1}^\intercal v \in [u, 2u)\}\right) \\
u_d^* &\triangleq  \argmax_{u\in [0, d]} \on{vol}_d\left(\{v\in [0, 1]^d: \mbf{1}^\intercal v \in [u, 2u)\}\right)\\
S^*_{d,n} &\triangleq  \{v\in [n]^d:\mbf{1}^\intercal v \in [u_d^*n, 2u_d^*n)\}.
\end{aligned}
\label{eq:cd*}
\end{equation}

Observe that $S_{d,n}^*$ is a sum-free set with asymptotic density $c_d^*$. Hence, $c_d \geq c_d^*$. It is conjectured that $c_d = c_d^*$, meaning that $S_{d,n}^*$ is asymptotically the largest sum-free subset of $[n]^d$. Because of this simple lower bound, the heart of the problem lies in the upper bound $c_d \leq c_d^*$. 

For $d=1$, we have $u_1^*=1/2$, so $c_1\geq c_1^*=1/2$. The corresponding upper bound follows from a straightforward argument pairing elements that sum to the maximum of the set.

For $d=2$, we have $u_2^*=4/5$, so $c_2\geq c_2^*=3/5$. Recently, Elsholtz and Rackham \cite{ER17} provided the matching upper bound through detailed analysis of the upper convex hull of a sum-free set. Liu, Wang, Wilkes, and Yang \cite{LWWY23} established the stability of this maximum. However, extending these geometric methods to higher dimensions is intractable due to increasing complexity.

For $d\ge 3$, there remains a notable gap between the upper and lower bounds of $c_d$. For $d=3$, we have $u_3^* = (15-\sqrt{15})/10$, so $c_3 \geq c_3^*=(10+\sqrt{15})/20\approx 0.69$. Cameron \cite{Cam02} provides the best current upper bound of $c_3\leq 2/e\approx 0.74$. For $d=4$ the best known bounds are $0.76 \approx c_4^*\le c_4\le 0.96 $  \cite{katz,ER17}.  

\subsection{Main results}

\label{sec:results}

We determine $c_d$ for $d \in \{ 3,4 \}$. Our approach is based on a certain linear programming relaxation and it gives us a framework to address the question in higher dimensions.

\begin{theorem}\label{thm:main}
Let $d\in \{3, 4\}$. The size of the largest sum-free subset $S_{d,n}\subset [n]^d$ is
\[\left|S_{d,n}\right| = c_d^*\, n^d+O\left(n^{d-1/2}\right) \quad\text{where}\quad c_d^* = \max_{u\in [0, d]} \on{vol}_d\left(\{v\in [0, 1]^d: \mbf{1}^\intercal v \in [u, 2u)\}\right). \]
\end{theorem}

As a corollary, we can also obtain the following analog in the continuous setting.

\begin{corollary}\label{cor:main-cont}
Let $d\in \{3, 4\}$.
The largest measurable sum-free subset of $[0,1]^{d}$ is 

\begin{equation}\label{eq:S*}
S^* = \{ v \in [0,1]^d : \mbf{1}^\intercal v \in [u_d^*, 2u_d^*)\}
\end{equation}

and it has Lebesgue measure $c_d^*$, where $u_d^*$ and $c_d^*$ are defined in \ref{eq:cd*}.
\end{corollary}

\subsection{Organization}

\cref{sec:reduction} presents the proof of \cref{thm:main} for arbitrary dimensions $d \geq 3$, assuming the existence of a dual weight function satisfying certain conditions as outlined \cref{prop:weight}. In following sections, we construct the function for dimensions $d \in \{3,4\}$.  

In \cref{sec:key} and \cref{sec:leftover}, we construct continuous analogue of the weight function. In \cref{sec:discr}, we discretize the continuous function, proving \cref{prop:weight}. In \cref{sec:conclusion}, we explore the extension of our framework to arbitrary dimensions.

\subsection{Notation}\label{sec:notation}

In the paper, all logs are natural and all constants in asymptotic notations are universal. When clear from the context, we will often drop the dimension subscript $d$. 

The calligraphic letters are reserved for polytopes (e.g. $\mc X$, $\mc Y$, $\mc Z$). For a polytope $\mc P \subset \mb R^d$ we define $\mc P^n \triangleq n \mc P \cap \mb Z^d.$ Furthermore, $\mc P^\circ$ will denote the interior.

For a set of points $T\subset\mb{R}^d$, let $\on{conv} T$ denote their convex hull. For $d\in\mb{N}$, define the standard simplex $\Delta^{d-1} := \on{conv} \{\mbf{e}_1, \dots , \mbf{e}_d\}\subset \mb{R}^d$, where $\mbf{e}_i$ is the $i$-th standard basis vector in $\mb{R}^d$. 

We will denote the uniform probability distribution on a set $S$ by $\U(S)$. When clear from the context, we will abuse notation to denote the Lebesgue measure and volume of appropriate dimension of the set $S$ by $|S|$. 

\subsection*{Acknowledgments}

This research was completed at the MIT SPUR program during the summer of 2023. We extend our profound gratitude towards our mentor Mehtaab Sawhney, for his invaluable guidance at every step of this project, and to the professor David Jerison, for directing the program and insightful discussions. We also thank the anonymous reviewers for their suggestions.

\section{Reduction to a weight Function}\label{sec:reduction}

In this section, we prove \cref{thm:main} modulo \cref{prop:weight}. We first show the lower bound, and then the reduction for upper bound. The reduction works for arbitrary dimensions $d \geq 3$.  

\subsection{The lower bound}\label{sec:lower}

Throughout this paper, we will interchange the volume of a polytope with the number of its lattice points, owing to the following standard fact from the geometry of numbers.

\begin{fact}\label{fact:geometry-of-numbers}
Let $\mc{P}\subset [0, 1]^d$ be a closed polytope with nonempty interior. Let $\mc{P}^n = n\mc P \cap \mb{Z}^d$, then 
\begin{equation*}
	\left|\mc{P}^n\right| = (1 + O(n^{-1}))|n\mc{P}| = n^d |\mc{P}| + O\left(n^{d-1}\right).
\end{equation*} 
\end{fact}

We now formally record the lower bound construction mentioned in \cref{sec:history}. 

\begin{proposition}[Lower Bound]\label{prop:lower}
Fix dimension $d\in \mb{N}$, recall $u_d^*$, $c_d^*$, and $S_{d,n}^*$ from \ref{eq:cd*}. Then $S_{d,n}^* \subset [n]^d$ is a sum-free subset of size $c_d^*n^d +O(n^{d-1})$.
\end{proposition}

\begin{proof}
    If $x, y \in S_{d,n}^*$, then $\mbf{1}^\intercal(x + y) \geq 2u_d^*n$, so $x + y \notin S_{d,n}^*$. Hence, $S_{d,n}^*$ is sum-free. By \cref{fact:geometry-of-numbers} we obtain that $|S_{d,n}^*|= c_d^*n^d +O(n^{d-1})$.
\end{proof}

\subsection{The upper bound}\label{sec:upper}

We begin with the following key one-dimensional observation.

\begin{lemma}\label{lem:nd-lines}
Let $\pi :\mb R^d \to \mb R^{d-1}$ be a projection that forgets the last coordinate. 
Let $S \subset [n]^d$ be a sum-free set and $x, y, z \in [n]^{d-1}$ with $x + y =z$. Then,
\begin{equation}\label{eq:nd-lines-bd}
   |S \cap \pi^{-1}(x)| + |S \cap \pi^{-1}(y)| + |S \cap \pi^{-1}(z)| \leq 2n.
\end{equation}
\end{lemma}

\begin{proof}
If $S \cap \pi^{-1}(z) = \varnothing$, then the conclusion follows immediately. Otherwise, let $m \in [n]$ be the largest integer such that $(z, m) \in S$. Since $S$ is sum-free, either $(x,k) \notin S$ or $(y, m - k) \notin S$ for each $0 \leq k \leq m$. Therefore, $|S| \leq 3n - (m+1) - (n-1-m) = 2n$.
\end{proof}

Under the projection $\pi : \mb R^d \to \mb R^{d-1}$, the image of $S_{d,n}^*$ naturally partitions the cube into five regions as defined as follows.
\begin{equation}\label{eq:ABCDE}
\begin{aligned}
\mc{A} &:= \{v\in [0,1]^{d-1}:\mbf{1}^\intercal v \in [0, u_{d}^*-1)\}; \\
\mc{B} &:= \{v\in [0,1]^{d-1}:\mbf{1}^\intercal v \in [u_{d}^*-1, u_d^*)\};\\
\mc{C} &:= \{v\in [0,1]^{d-1}:\mbf{1}^\intercal v \in [u_d^*, 2u_d^*-1)\};\\
\mc{D} &:= \{v\in [0,1]^{d-1}:\mbf{1}^\intercal v \in [2u_d^*-1, 2u_d^*)\};\\
\mc{E} &:= \{v\in [0,1]^{d-1}:\mbf{1}^\intercal v \in [2u_d^*, d-1]\}.
\end{aligned}
\end{equation}
Furthermore, optimality of $u_d^*$ implies that $|\mc B| = 2|\mc D|.$ Moreover, for $d \in \set{3,4}$ and we have $u_d^* > 1$ by \cref{fact:ud-comp}, so the regions are disjoint\footnote{In fact, for any $d \geq 3$ we have $u_d^* > 1$.}. Recall that by convention $\mc X^n \triangleq n \mc X \cap \mb{Z}^{d-1}$ and define  
\begin{equation}\label{eq:Omega}
    \Omega \triangleq  \set{(x,y,z) \in (\mc{B}^n \times \mc{B}^n \times \mc{D}^n) \cup (\mc{A}^n \times \mc{C}^n \times \mc{C}^n) \cup (\mc{C}^n \times \mc{C}^n \times \mc{E}^n) : x + y =z }.
\end{equation}

\begin{lemma}\label{lem:opt-lines}
Let $S_{d,n}^*$ be as in \cref{eq:cd*}. For any $(x,y,z) \in \Omega$,
\begin{equation*}
      |S_{d,n}^* \cap \pi^{-1}(x)| + |S_{d,n}^*\cap \pi^{-1}(y)| + |S_{d,n}^* \cap \pi^{-1}(z)| \geq 2n-1.
\end{equation*}
\end{lemma}

\begin{proof}
Suppose first that at least two elements of $\{x,y,z\}$ lie in $\mc{C}^n$, then the conclusion follows immediately as  $\pi^{-1}(\mc{C}^n) \subset S_{d,n}^*$. Otherwise, suppose $(x, y, z) \in \mc{B}^n \times \mc{B}^n \times \mc{D}^n$. Then,\begin{align*}
    |S_{d,n}^* \cap \pi^{-1}(x)| &= \fl{n - (u_d^*n -\mbf{1}^\intercal x)},  \\
    |S_{d,n}^* \cap \pi^{-1}(y)| &= \fl{n - (u_d^*n -\mbf{1}^\intercal y)}, \\
    |S_{d,n}^* \cap \pi^{-1}(z)| &= \ceil{2u_d^*n - \mbf{1}^\intercal z} .
\end{align*}
Hence, $ |S_{d,n}^* \cap \pi^{-1}(x)| +  |S_{d,n}^* \cap \pi^{-1}(y)| +  |S_{d,n}^* \cap \pi^{-1}(z)| \geq 2n-1$. 
\end{proof}

We will now state the key proposition, which will be proved in the following chapters.

\begin{proposition}[Existence of Weight Function]\label{prop:weight}
Let $d \in \{3,4\}$. There exists a function $w : \mb Z^{d-1} \times \mb Z^{d-1} \times \mb Z^{d-1} \to \mb R_{+}$ such that the following three conditions hold:
\begin{align}
\sum_{(x, y, z)\not\in \Omega} w(x,y,z) & = O\left(n^{d-3/2}\right)\label{eq:weight-1}, \\ 
\sum_{v\in \mc{A}^n \cup \mc{B}^n\cup \mc{D}^n\cup \mc{E}^n} \left| W(v) - 1 \right| & = O\left(n^{d-3/2}\right) \label{eq:weight-2},\\
\sum_{v\in\C} |W(v) - 1|_+&= O\left(n^{d-3/2}\right) \label{eq:weight-3},
\end{align}
where $|x|_+ = \max \set{x, 0}$ and $W:\mb{Z}^{d-1}\to \mb{R}_+$ is defined by
\begin{equation}\label{eq:def-W}
W(v) := \sum_{x,  y \in \mb{Z}^{d-1}}w(v, x, y)+w(x, v, y)+w(x, y, v).
\end{equation}
\end{proposition}

Roughly speaking, \cref{prop:weight} says up to lower order terms
\begin{enumerate}
	\item $w$ is supported on $\Omega$ 
	\item $W$ is identity on $\mc A \cup \mc B \cup \mc D \cup \mc E$
	\item $W$ is at most one on $\mc C$.
\end{enumerate} 

We now establish \cref{thm:main} assuming \cref{prop:weight}.

\begin{proof}[Proof of \cref{thm:main}.]

    Let $S \subset [n]^d$ be arbitrary sum-free set. For brevity, for each $v \in \mb{Z}^{d-1}$, define $\lambda(v)  \triangleq  |S \cap \pi^{-1}(v)|$ and $\lambda^*(v) \triangleq |S_{d,n}^* \cap \pi^{-1}(v)|$. By \cref{lem:nd-lines,lem:opt-lines},
    
    \begin{equation}\label{eq:ineq-triple}
        \lambda(x) + \lambda(y) + \lambda(z) \leq \lambda^*(x) + \lambda^*(y) + \lambda^*(z) + 1 \text{ for each }  (x,y,z) \in \Omega.
    \end{equation}
    
    Furthermore, for each $v \in \mc{C}^n$ we have $\lambda^*(v) = n$, so 
    \begin{equation}\label{eq:ineq-single}
    	 \lambda(v) \leq \lambda^*(v) \text{ for each } v \in \mc{C}^n
    \end{equation}

     We will combine \ref{eq:ineq-triple} and \ref{eq:ineq-single} with suitable weights to deduce \cref{thm:main}. Let $w$ and $W$ be as in \cref{prop:weight}. Since $\Omega\subset [n]^{d-1}\times [n]^{d-1}\times [n]^{d-1}$, we can apply property \ref{eq:weight-1} to get
     
    \begin{equation}\label{eq:app-w-a}
        \sum_{v \not\in [n]^{d-1}} W(v) \le 3 \times \sum_{(x, y, z)\not\in \Omega} w(x,y,z) = O\left(n^{d-3/2}\right).
    \end{equation}
    
Thus, the weight outside of $[n]^{d-1}$ is small. We also bound the weight in $\T$
as follows:
\begin{equation}\label{eq:app-w-a2}
\sum_{ (x,y,z) \in \Omega} w(x,y,z) \le \sum_{v\in [n]^{d-1}} W(v) \le n^{d-1}+\sum_{v\in [n]^{d-1}} \left|W(v) -1\right|_+ \le O\left(n^{d-1}\right).
\end{equation}

In the last step we partitioned $[n]^{d-1}$ into $\A\cup\B\cup\C\cup\D\cup\E$ and applied \ref{eq:weight-2} and \ref{eq:weight-3}. Next, we make the key observation
\begin{equation}
\begin{aligned}\label{eq:comp-sum-1}  
     \sum_{ x,y,z \in \mb{Z}^{d-1}} w(x,y,z) &(\lambda(x) + \lambda(y) + \lambda(z)) +  \sum_{v \in \mc{C}^n} (1-W(v)) \lambda(v) = 
    \\ & = \sum_{v \in \mb{Z}^{d-1}} W(v)\lambda(v) +  \sum_{v \in \mc{C}^n} (1-W(v)) \lambda(v)\\
    &= \sum_{v \notin [n]^{d-1}} W(v) \lambda(v) + \sum_{v \in [n]^{d-1} \setminus \mc{C}^n} (W(v)-1) \lambda(v) + \sum_{v \in [n]^{d-1}} \lambda(v) \\
    &= O\left(n^{d-1/2}\right) + \sum_{v \in [n]^{d-1}} \lambda(v).
\end{aligned}
\end{equation}

    In the last step, we applied \cref{eq:app-w-a,eq:weight-2} and the trivial bound $\lambda (v)\le n$. On the other hand, by \cref{eq:weight-3} we   bound
\begin{equation}\label{eq:comp-sum-2} 
\begin{aligned}
\sum_{v \in \mc{C}^n} (1-W(v)) \lambda(v) & = \sum_{v \in \mc{C}^n} |1-W(v)|_+\lambda(v) - \sum_{v \in \mc{C}^n} |W(v)-1|_+  \lambda(v) 
\\ & = \left( \sum_{v \in \mc{C}^n} |1-W(v)|_+\lambda(v) \right)-O\left(n^{d-1/2}\right).   
\end{aligned}   
\end{equation}
Combining \cref{eq:comp-sum-1,eq:comp-sum-2}, and applying the same computation to $S_{d,n}^*$ instead of $S$, we conclude
     \begin{align*}
         |S| &= \sum_{v \in [n]^{d-1}} \lambda(v) \\
        &=O\left(n^{d-1/2}\right) + \sum_{ (x,y,z) \in \Omega} w(x,y,z) (\lambda(x) + \lambda(y) + \lambda(z)) + \sum_{v \in \mc{C}^n} |1-W(v)|_+ \lambda(v) \\
        &\leq O\left(n^{d-1/2}\right) + \sum_{ (x,y,z) \in \Omega} w(x,y,z) (\lambda^*(x) + \lambda^*(y) + \lambda^*(z)+1) + \sum_{v \in \mc{C}^n} |1-W(v)|_+ \lambda^*(v) \\
        &=  O\left(n^{d-1/2}\right) + \sum_{ (x,y,z) \in \Omega} w(x,y,z) + \sum_{v \in [n]^{d-1}} \lambda^*(v) \\
        &= O\left(n^{d-1/2}\right) + |S_{d,n}^*|.
    \end{align*}
The inequality step holds by \cref{eq:ineq-triple} and \cref{eq:ineq-single}.
\end{proof}

Note that the above reduction to \cref{prop:weight} works in arbitrary dimensions $d\geq 3$. Hence, the missing piece to arbitrary dimensions is only the weigh function in higher dimensions. In following chapters, we verify the \cref{prop:weight} for dimensions $3$ and $4$.

\section{Joint Mixability and Compatibility}
In this section we define joint mixability and compatibility. These notions address whether it is possible to couple distributions while maintaining certain joint additive structure.

Such couplings are not new to  additive combinatorics. Peabody \cite{peabody} demonstrated the joint mixability of certain distributions to prove the conjecture of Kleinberg-Speyer-Sawin \cite{kss} on tri-colored sum-free sets. Following this, Lov\'{a}sz and Sauermann \cite{LS19} studied joint mixability in the context of $k$-multicolored sum-free sets in $\mb{Z}_m^n$. Joint mixability also arises in empirical risk and optimization problems \cite{WangWang,MR3385976}.

\begin{definition}[\cite{peabody,WangWang,MR3385976}]\label{def:jm}
We say that a triple of distributions $(\mu_1, \mu_2, \mu_3)$ on $\mb{R}^d$ is
\begin{itemize}
\item \emph{Jointly mixable} if there exists a coupling $(X_1, X_2, X_3)$ with marginals $X_i\sim \mu_i$ for  each $i\in [3]$, such that $X_1+X_2+X_3 = \mb{E}(X_1+X_2+X_3)$ almost surely. 
	\item \emph{Compatible} if there exists a coupling $(X_1, X_2, X_3)$ with marginals $X_i\sim \mu_i$ for each $i\in [3]$, such that $X_1+X_2=X_3$;
	\end{itemize}
We say that the coupling $(X_1, X_2, X_3)$ \emph{witnesses} joint mixability / compatibility.\end{definition}

In the context of our problem, we will prove compatibility of uniform distributions on certain polytopes. We will then use the probability density function of the couplings to construct the weight function and prove \cref{prop:weight}. In the proposition below, we summarize the known sufficient conditions for joint mixability / compatibility.

\begin{proposition}\label{prop:WW}
Let $(\mu_1, \mu_2, \mu_3)$ be piecewise continuous probability density functions on $\mb{R}$. Let $a_i \triangleq  \inf \on{supp} \mu_i$, $b_i \triangleq \sup \on{supp} \mu_i$, and $c_i \triangleq  b_i-a_i$ . Then, $(\mu_1, \mu_2, \mu_3)$ are jointly mixable if any of the following conditions hold.
\begin{enumerate} 
\item {\cite[Theorems 3.2] {WangWang}} For each $1 \leq i \leq 3$, function $\mu_i$ is decreasing on $[a_i, b_i]$ and
\begin{equation}\label{eq:mean-ineq}
\max_{i \in [3]} c_i +\sum_{i=1}^3a_i\le\sum_{i=1}^3\mb{E}\mu_i\le -\max_{i \in [3]} c_i+\sum_{i=1}^3b_i.
\end{equation}
\item {\cite[Theorem 3.1] {WangWang}} 
For each $1 \leq i \leq 3$, distribution $\mu_i = \U([a_i, b_i])$ and
\begin{equation*}
	c_1+c_2+c_3\le 2\max(c_1, c_2, c_3)
\end{equation*}
\item {\cite[Theorem 3.4] {WangWang}} 
For each $1 \leq i \leq 3$, distribution $\mu_i$ is symmetric unimodal with zero mode. Moreover, for all $K \in (0, 1/2)$ we have  
\begin{equation*}
	\ell_1(K)+\ell_2(K)+\ell_3(K)\le 2\max(\ell_1(K), \ell_2(K), \ell_3(K)),
\end{equation*}
where $\ell_i(K)\geq 0$ is unique real number defined by
\begin{equation}\label{eq:dagger}
\int_0^{\ell_i(K)}\mu_i(x) - \mu_i(\ell_i(K)) \, dx=K.
\end{equation}
\end{enumerate}
\end{proposition}

\section[The Coupling Supported on $\mc{B}\times \mc{B}\times \mc{D}$]{The Coupling Supported on \texorpdfstring{$\mc{B}\times \mc{B}\times \mc{D}$}{B x B x D}}\label{sec:key}
In this section we will show compatibility of $(\U(\mc B), \U(\mc B), \U(\mc D))$ as summarized below.

\begin{proposition}\label{prop:BBD}
Let $d\in \{3,4\}$. There exist random vectors $(X,Y,Z)$ such that marginally $X\sim \U(\mc{B})$, $Y\sim \U(\mc{B}),$ and $Z\sim \U(\mc{D})$, and jointly $X+Y = Z$ almost surely. 
\end{proposition}

\subsection[Construction for $d=3$]{Construction for \texorpdfstring{$d=3$}{d = 3}}\label{sec:3-key}\begin{proof}[{Proof of \cref{prop:BBD} with $d=3$}]
Relabel $\mc{X}=\mc{Y}= \mc B$ and $\mc{Z}=\mc D$. Let us reparametrize the regions as
\[ \mc{Z} = \on{conv}\left\{(1, 2u_3^*-2), (2u_3^*-2, 1), (1,1)\right\}\]
and partition $\mc{X}=\bigcup_{i=0}^3 \mc{X}_i = \mc{Y}=\bigcup_{i=0}^3 \mc{Y}_i$ by defining
 \begin{align*}
 &
 \begin{aligned}
 	\mc{X}_0 &=  \on{conv}\left\{(0, u_3^* -1), (u_3^* -1,1), (1,0)\right\}, \\
 	\mc{X}_1 &=  \on{conv}\left\{(0, u_3^* -1), (u_3^* -1,0), (1,0)\right\}, \\
 	\mc{X}_2 &=  \on{conv}\left\{(1, u_3^* -1), (u_3^* -1,1), (1,0)\right\}, \\
 	\mc{X}_3 &=  \on{conv}\left\{(0, u_3^* -1), (u_3 -1,1), (1,0)\right\}, \\
 \end{aligned} 
 \begin{aligned}
 	\mc{Y}_0 &=  \on{conv}\left\{(1, u_3^* -1), (u_3^* -1,0), (0,1)\right\}, \\
 	\mc{Y}_1 &=  \on{conv}\left\{(1, u_3^* -1), (u_3^* -1,1), (0,1)\right\}, \\
 	\mc{Y}_2 &=  \on{conv}\left\{(0, u_3^* -1), (u_3^* -1,0), (0,1)\right\}, \\
 	\mc{Y}_3 &=  \on{conv}\left\{(1, u_3^* -1), (u_3^* -1,0), (0,1)\right\}. \\
 \end{aligned}
 \end{align*}

Observe that we can match vertices of the same colors, and extend by convex combinations. For each $\lambda = (\lambda_1, \lambda_2, \lambda_3) \in \Delta^2$, we let $Z(\lambda) \triangleq \lambda_1(1, 2u_3^*-2) + \lambda_2(2u_3^*-2, 1) + \lambda_3(1,1)$ and
 \begin{align*}
 &
 \begin{aligned}
 	X_0(\lambda) &\triangleq  \lambda_1(0, u_3^* -1) + \lambda_2(u_3^* -1,1) + \lambda_3(1,0),\\
 	X_1(\lambda) &\triangleq \lambda_1(0, u_3^* -1) + \lambda_2(u_3^* -1,0) + \lambda_3(1,0),\\
 	X_2(\lambda) &\triangleq \lambda_1(1, u_3^* -1) + \lambda_2(u_3^* -1,1) + \lambda_3(1,0),\\
 	X_3(\lambda) &\triangleq \lambda_1(0, u_3^* -1) + \lambda_2(u_3^* -1,1) + \lambda_3(1,0),\\
 \end{aligned} 
 &&
 \begin{aligned}
 	Y_0(\lambda) &\triangleq \lambda_1(1, u_3^* -1) + \lambda_2(u_3^* -1,0) + \lambda_3(0,1),\\
 	Y_1(\lambda) &\triangleq \lambda_1(1, u_3^* -1) + \lambda_2(u_3^* -1,1) + \lambda_3(0,1),\\
 	Y_2(\lambda) &\triangleq \lambda_1(0, u_3^* -1) + \lambda_2(u_3^* -1,0) + \lambda_3(0,1),\\
 	Y_3(\lambda) &\triangleq \lambda_1(1, u_3^* -1) + \lambda_2(u_3^* -1,0) + \lambda_3(0,1).\\
 \end{aligned}
 \end{align*}

\begin{tikzpicture}[scale=3]  
  \pgfmathsetmacro{\u}{(15-sqrt(15))/10}

  \begin{scope}[shift={(0,0)}]
  
  \draw[thick] (0,0) -- (1,0) -- (1,1) -- (0,1) -- cycle; 
  
  \fill[gray!25] (0,\u-1) -- (\u - 1,0) -- (1,0) -- (1, \u - 1) -- (\u-1,1) -- (0,1) -- cycle; 

  \coordinate (B) at (\u-1,0);
  \coordinate (C) at (0,\u-1);
  \coordinate (D) at (1,0);
  \coordinate (E) at (1,\u-1);
  \coordinate (F) at (\u-1,1);
  \coordinate (G) at (0,1);

  \draw[thick] (C) -- (D) -- (F) -- cycle;
  \draw[thick] (B) -- (D) -- (E) -- (F) -- (G) -- (C) -- cycle;  
  
  \node at ({(0.35}, {0.35}) {$\mc{X}_0$};
  \node at ({0.5}, {-0.1}) {$\mc{X}_1$};
  \node at ({0.65}, {0.65}) {$\mc{X}_2$};
  \node at ({-0.11}, {0.5}) {$\mc{X}_3$};
  

  \node[circle,fill=green!50!black,inner sep=2pt] at (C) {}; 
  \node[circle,fill=blue,inner sep=2pt] at (D) {}; 
  \node[circle,fill=red,inner sep=2pt] at (F) {}; 
\end{scope}

\begin{scope}[shift={(2,0)}] 
  
  \draw[thick] (0,0) -- (1,0) -- (1,1) -- (0,1) -- cycle; 

  \fill[gray!25] (0,\u-1) -- (\u - 1,0) -- (1,0) -- (1, \u - 1) -- (\u-1,1) -- (0,1) -- cycle; 

  \coordinate (B) at (\u-1,0);
  \coordinate (C) at (0,\u-1);
  \coordinate (D) at (1,0);
  \coordinate (E) at (1,\u-1);
  \coordinate (F) at (\u-1,1);
  \coordinate (G) at (0,1);

  \draw[thick] (B) -- (E) -- (G) -- cycle;
  \draw[thick] (B) -- (D) -- (E) -- (F) -- (G) -- (C) -- cycle;
  
  \node at ({(0.35}, {0.35}) {$\mc{Y}_0$};
  \node at ({0.5}, {-0.1}) {$\mc{Y}_3$};
  \node at ({0.65}, {0.65}) {$\mc{Y}_1$};
  \node at ({-0.11}, {0.5}) {$\mc{Y}_2$};

  \node[circle,fill = red,inner sep=2pt] at (B) {};
  \node[circle,fill = green!50!black,inner sep=2pt] at (E) {}; 
  \node[circle,fill = blue,inner sep=2pt] at (G) {};

\end{scope}

  \begin{scope}[shift={(4,0)}]
  \draw[thick] (0,0) -- (1,0) -- (1,1) -- (0,1) -- cycle; 

  \fill[gray!25] (2*\u-2, 1) -- (1,1) -- (1,2*\u-2) -- cycle; 

  \coordinate (H) at (1,2*\u-2);
  \coordinate (I) at (2*\u-2,1);
  \coordinate (J) at (1,1);
	
  \draw[thick] (H) -- (I) -- (J) -- cycle;
  
  \node at ({0.75, 0.75}) {$\mc{Z}$};

  \node[circle,fill = green!50!black,inner sep=2pt] at (H) {}; 
  \node[circle,fill = red,inner sep=2pt] at (I) {}; 
  \node[circle,fill = blue,inner sep=2pt] at (J) {}; 
  \end{scope}
\end{tikzpicture}
 
Observe that $X_i(\lambda) + Y_i(\lambda) = Z_i(\lambda)$. We define the  coupling $(X, Y, Z)$ as follows:
\begin{itemize}
	\item[-] Sample coefficients $\lambda \sim \U\left(\Delta^2\right);$
	\item[-] Independently sample index $I\in \{0, 1, 2, 3\}$ with $\mb{P}(I=i) \propto |\mc X_i | = |\mc{Y}_i|;$
	\item[-] Set $(X, Y, Z)=(X_I(\lambda), Y_I(\lambda), Z(\lambda))$
\end{itemize}
Clearly, $X+Y=Z$ almost surely and the marginals are uniform on their domains.
\end{proof}

\subsection[Construction for $d=4$]{Construction for \texorpdfstring{$d=4$}{d = 4}}\label{sec:4-key} Now we will take a different strategy constructing coupling  coordinate-by-coordinate using \cref{prop:WW}.
\begin{definition}\label{def:one-dim-distr}
 Let $X_1, X_2, X_3, X_4 \sim \U ([0, 1])$ be independent and identically distributed.
 \begin{itemize}
	\item Define $\mu_X = \mu_Y$ to be the conditional probability density functions of 
\[ X_1+X_2 \mid \{X_1+X_2+ X_3 + X_4= u_4^*\}.\]
	\item Define $\mu_Z$ to be the conditional probability density function of 
\[ X_1+X_2 \mid \{X_1+X_2+ X_3 + X_4= 2u_4^*\}.\]
\end{itemize}
\end{definition}

Now we present the key technical lemma.
\begin{lemma}\label{lem:gabc-jm}
$(\mu_X,\mu_Y,\mu_Z)$ are compatible.
\end{lemma}
\begin{proof}[Proof Sketch]
Observe that $\mu_X$, $\mu_Y$, and $\mu_Z$ are all symmetric and uni-modal. Therefore, it suffices to verify the third condition given in \cref{prop:WW}, which we defer to \cref{sec:mathematica}.
\end{proof}

We now main coupling lemma about critical slices of $[0, 1]^4$.
Define $\mc{P}^d_t = \set{v \in [0,1]^d : 1^\intercal  v = t }.$ 

\begin{lemma}\label{lem:4abc-coupling}
$(\U (\mc P^4_{u_4^*}), \U (\mc P^4_{u_4^*}), \U (\mc P^4_{2u_4^*}))$ are compatible.
\end{lemma}
\begin{proof}
By \cref{lem:gabc-jm}, triple  $\left(\mu_X,\mu_Y,\mu_Z\right)$ is compatible with witnesses $(\alpha, \beta, \gamma)$. This means that $\alpha+\beta = \gamma$ and $(u_4^*-\alpha)+(u_4^*-\beta)=2u_4^*-\gamma$ almost surely. By the second condition of \cref{prop:WW},
\begin{equation*}
	(\U (\mc P^2_{\alpha}), \U (\mc P^2_{\beta}), \U (\mc P^2_{\gamma}))\quad\text{and}\quad (\U (\mc P^2_{u_4^*-\alpha}), \U (\mc P^2_{u_4^*-\beta}), \U (\mc P^2_{2u_4^*-\gamma})).
\end{equation*}
 are both compatible with witnesses $((X_1, X_2), (Y_1, Y_2), (Z_1, Z_2))$ and $((X_3, X_4), (Y_3, Y_4), (Z_3, Z_4))$, respectively. We can now form the coupling 
\[(X, Y, Z) = \left((X_1, X_2, X_3, X_4), (Y_1, Y_2, Y_3, Y_4), (Z_1, Z_2, Z_3, Z_4)\right).\]
Observe that by definition of $\mu_X,\mu_Y, \mu_Z$, the marginals are precisely $\U (\mc P^4_{u_4^*})$, $\U (\mc P^4_{u_4^*})$, and $\U (\mc P^4_{2u_4^*})$. Moreover, by construction $X_i+Y_i=Z_i$ almost surely for each $1 \leq i \leq 4$, so $(X, Y, Z)$ witnesses the claimed compatibility.
\end{proof}

\begin{proof}[{Proof of \cref{prop:BBD} with $d=4$}] By \cref{lem:4abc-coupling}, the triple $(\U (\mc P^4_{u_4^*}), \U (\mc P^4_{u_4^*}), \U (\mc P^4_{2u_4^*}))$ are compatible. Note that $\mc B = \pi (\mc P^4_{u_4^*})$ and $\mc D = \pi (\mc P^4_{2u_4^*})$. Since the projection $\pi$ preserves compatibility, we are done. 
\end{proof}

\section[The Couplings Supported on $\mc{A}\times\mc{C}\times\mc{C}$ and  $\mc{C}\times\mc{C}\times\mc{E}$]{The Couplings Supported on \texorpdfstring{$\mc{A}\times\mc{C}\times\mc{C}$}{A x C x C} and  \texorpdfstring{$\mc{C}\times\mc{C}\times\mc{E}$}{C x C x E}}\label{sec:leftover}
We construct the couplings supported on $\mc{A}\times\mc{C}\times\mc{C}$ and $\mc{C}\times\mc{C}\times\mc{E}$ in the continuous case. It turns out that the following proposition is sufficient for \cref{prop:weight}. The constructions will have marginals uniform on some simplices. Observe that in $d = 3$ the region $\mc{A}$ is a simplex and $\mc{E}$ is empty, while for $d =4$ both are simplices.

\begin{proposition}\label{prop:cont-leftover}
The following are true 
\begin{itemize}
    \item For $d=3$, there are simplices $\mc{X}_1, \mc{Y}_1, \mc{Z}_1\subset [0, 1]^{d-1}$ such that 
\begin{equation}\label{eq:leftover-regions-3}
\mc{X}_1=\mc{A} \text{, and }  \mc{Y}_1, \mc{Z}_1\subset \mc{C} \text{ satisfies } \mc{Z}_1^\circ\cap \mc{Y}_1^\circ = \varnothing,
\end{equation}
triple $\left(\U\left(\mc{X}_1\right), \U\left(\mc{Y}_1\right), \U\left(\mc{Z}_1\right)\right)$ is compatible, and that
\begin{equation}\label{eq:cont-room-c-3}
\frac{|\mc{X}_1|}{|\mc{Z}_1|} < 1 \quad\text{and}\quad \frac{|\mc{X}_1|}{|\mc{Y}_1|} \leq 1.
\end{equation}

\item For $d=4$, there are simplices $\mc{X}_1, \mc{Y}_1, \mc{Z}_1, \mc{X}_2, \mc{Y}_2, \mc{Z}_2\subset [0, 1]^{d-1}$ such that 
\begin{equation}\label{eq:leftover-regions}
\mc{X}_1=\mc{A}, \mc{Z}_2=\mc{E} \text{, and }  \mc{Y}_1, \mc{Z}_1, \mc{X}_2, \mc{Y}_2\subset \mc{C} \text{ satisfies } \mc{Z}_1^\circ \cap\left(\mc{Y}_1\cup \mc{X}_2\cup \mc{Y}_2\right)^\circ = \varnothing,
\end{equation}
triples $\left(\U\left(\mc{X}_1\right), \U\left(\mc{Y}_1\right), \U\left(\mc{Z}_1\right)\right)$ and $\left(\U\left(\mc{X}_2\right), \U\left(\mc{Y}_2\right), \U\left(\mc{Z}_2\right)\right)$ are compatible, and that
\begin{equation}\label{eq:cont-room-c}
\frac{|\mc{X}_1|}{|\mc{Z}_1|} < 1 \quad\text{and}\quad \frac{|\mc{X}_1|}{|\mc{Y}_1|}+\frac{|\mc{Z}_2|}{|\mc{X}_2|}+\frac{|\mc{Z}_2|}{|\mc{Y}_2|} < 1.
\end{equation}
\end{itemize}
\end{proposition}

\subsection{Parametrizing simplices}
All the $d$-dimensional simplices we seek will have one vertex lying on the line spanned by all-ones vector $\mbf{1}\in \mb{R}^d$, and the other $d$ vertices have coordinates that are cyclic permutations of each other. We parametrize them as follows: for $t \in [0, d]$ and $v\in [0, 1]^d$, let
\begin{equation}\label{eq:sigma}
\Sigma_{d}(t , v):= \on{conv}\left\{\frac{t  \mbf{1}}{d}, \;\sigma^{i}_{d}(v):0\le i\le d-1 \right\}\subset [0, 1]^d,
\end{equation}
where $\sigma_{d}: (v_1, \dots, v_{d})\mapsto (v_2, \dots, v_{d}, v_1)$ and the exponent $i$ denotes the $i$-fold composition. 
We begin with the following lemma that computes volumes of such simplices up to a fixed constant $C_d$.

\begin{fact}\label{lem:vol-sigma}
For $d \in \set{2,3}$, there exists some constant $C_{d}>0$ such that
\begin{equation}\label{eq:vol-sigma}
\on{vol}_{d}\Sigma_{d}(t , v) = C_{d} \left|t -\mbf{1}^\intercal v\right|\cdot  \left\Vert v-\sigma_{d}(v)\right\Vert_2^{d-1}.
\end{equation} 
\end{fact}
\begin{proof}
We view $\mathrm{conv}\{ \sigma_{d}^i(v)\}_{i=0}^{d-1}$ as the base of the simplex. In $d = 2$ this is a segment and in $d = 3$ equilateral triangle, both lying in affine hyperplane perpendicular to $\mbf{1}$ centered on $\mbf{1}$. Furthermore, $\sigma_d$ rotates the base around $\mbf{1}$ with $180^\circ$ in $d = 2$ and $120^\circ$ in $d = 3$. 
For this reason, the area of the base is proportional to $ \left\Vert v-\sigma_{d}(v)\right\Vert_2^{d-1}$, while the height is proportional to $|t - 1^\intercal v|$.
\end{proof}

We give a sufficient condition for the compatibility of uniform distributions on such simplices.

\begin{lemma}\label{lem:jm-simplices}
For any $a, b, c \in [0, d]$ and $x, y, z\in [0, 1]^d$ such that $a+b=c$ and $x+y=z$, the distributions
$\left(\U\left(\Sigma_d(a, x)\right), \U\left(\Sigma_d(b, y)\right), \U\left(\Sigma_d(c, z)\right)\right)$ are compatible.
\end{lemma}
\begin{proof}
For any $t\in[0, d]$ and $v\in [0,1]^d$, consider the affine map $F_{d}(t, v):\Delta^{d} \to \Sigma_d(t, v)$ given by 
\begin{equation}\label{eq:Ftv}
[F_{d}(t, v)](\mu_0, \dots, \mu_{d}) = \left(\frac{t\mu_d}{d}\right)\mbf{1}+ \sum_{i=0}^{d-1}\mu_i \sigma_{d}^i(v).
\end{equation}
Now, we couple the distributions as follows: take $\mu \in\U\left(\Delta^d\right)$ and let 
\[ (X, Y, Z)=\left([F_d(a, x)](\mu), [F_d(b, y)](\mu), [F_d(c, z)](\mu)\right).\]
Observe that marginals are uniform. That is, $[F_{d}(t, v)](\mu)\sim \U\left(\Sigma_d(t, v)\right)$ for every $t, v$ when $\mu\sim\U\left(\Delta^d\right)$, since $F$ induces a constant probability density on the simplex. It now suffices to prove $X+Y=Z$ almost surely. By linearity of $\sigma^i_d$, we have that for all $\mu\in \Delta^d$,
\[
[F_d(a, x)](\mu)+ [F_d(b, y)](\mu)  = \left(\frac{(a+b)\mu_d}{d}\right)\mbf{1}+ \sum_{i=0}^{d-1}\mu_i \sigma_{d}^i(x+y)
 = [F_d(c, z)](\mu). \qedhere
\]
\end{proof}
We now construct simplices $\mc{Y}_1, \mc{Z}_1, \mc{X}_2, \mc{Y}_2$ for $d\in\{3, 4\}$ and prove \cref{prop:cont-leftover}. To check that $\mc{Y}_1, \mc{Z}_1, \mc{X}_2, \mc{Y}_2\subset \mc{C}$ such that $\mc{Z}_1\cap\left(\mc{Y}_1\cup \mc{X}_2\cup \mc{Y}_2\right) = \varnothing$, we make the following trivial observation: 
\[ \mbf{1}^\intercal \left[\Sigma_{d-1}(t , v)\right] = [t, \mbf{1}^\intercal v],\]
where we let $\mbf{1}^\intercal S = \{\mbf{1}^\intercal s :s\in S\}$ for $S\subset \mb{R}^d$. By convexity, $\Sigma_{d-1}(t , v)\subset \mc{C}$ if $t , \mbf{1}^\intercal v\in [u_d, 2u_d-1]$.
The joint mixability of the distributions and \cref{eq:cont-room-c} will follow \cref{lem:jm-simplices,lem:vol-sigma}, respectively.

\subsection[Construction for $d=3$]{Construction for \texorpdfstring{$d = 3$}{d = 3}}\label{sec:3-leftover}
Here, $\mc{A}, \mc{C}\subset [0, 1]^2$ and $u_3^*\approx 1.112$ by \cref{fact:ud-comp}.

\begin{proof}[Proof of \cref{prop:cont-leftover} with $d=3$.]
We define the following regions of $\mb{R}^2$ and compute
 \begin{align*}
 \mc{X}_1 = \mc{A} = \Sigma_2\left(0, (u_3^*-1, 0)\right) & \implies \mbf{1}^\intercal \mc{X}_1 = [0, u_3^*-1],\\
 \mc{Y}_1 = \Sigma_2\left(\frac{3u_3^*-1}{2}, (u_3^*-1, 1)\right) & \implies \mbf{1}^\intercal \mc{Y}_1 =\left[u_3^*, \frac{3u_3^*-1}{2} \right], \\
 \mc{Z}_1 = \Sigma_2\left(\frac{3u_3^*-1}{2}, (2u_3^*-2, 1)\right) &\implies \mbf{1}^\intercal\mc{Z}_1 =\left[\frac{3u_3^*-1}{2}, 2u_3^*-1 \right].
 \end{align*}

\begin{center}
\begin{tikzpicture}[scale=3]

  \pgfmathsetmacro{\u}{(15-sqrt(15))/10}
  \pgfmathsetmacro{\um}{\u-1}         
  \pgfmathsetmacro{\t}{(3*\u-1)/2}    
  \pgfmathsetmacro{\th}{\t/2}         
  \pgfmathsetmacro{\vB}{2*\u-2}       

  \draw[thick] (0,0) -- (1,0) -- (1,1) -- (0,1) -- cycle;

  \coordinate (X1A) at (0,0);      
  \coordinate (X1B) at (\um,0);    
  \coordinate (X1C) at (0,\um);    
  \fill[gray!25] (X1A) -- (X1B) -- (X1C) -- cycle;
  \draw[thick] (X1A) -- (X1B) -- (X1C) -- cycle;
  \node at (0.15,0.15) {$\mc{X}_1$};

  \node[circle,fill=black,inner sep=2pt]            at (X1A) {};
  \node[circle,fill=black,inner sep=2pt]             at (X1B) {};
  \node[circle,fill=black!50!black,inner sep=2pt]  at (X1C) {};

  \coordinate (Y1A) at (\th,\th);   
  \coordinate (Y1B) at (\um,1);     
  \coordinate (Y1C) at (1,\um);     
  \fill[gray!25] (Y1A) -- (Y1B) -- (Y1C) -- cycle;
  \draw[thick] (Y1A) -- (Y1B) -- (Y1C) -- cycle;
  \node at (0.48,0.48) {$\mc{Y}_1$};

  \node[circle,fill=black,inner sep=2pt]            at (Y1A) {};
  \node[circle,fill=black,inner sep=2pt]             at (Y1B) {};
  \node[circle,fill=black!50!black,inner sep=2pt]  at (Y1C) {};

  \coordinate (Z1A) at (\th,\th);  
  \coordinate (Z1B) at (\vB,1);   
  \coordinate (Z1C) at (1,\vB);    
  \fill[gray!25]  (Z1A) -- (Z1B) -- (Z1C) -- cycle;
  \draw[thick] (Z1A) -- (Z1B) -- (Z1C) -- cycle;
  \node at (0.78,0.70) {$\mc{Z}_1$};

  \node[circle,fill=black,inner sep=2pt]            at (Z1A) {};
  \node[circle,fill=black,inner sep=2pt]             at (Z1B) {};
  \node[circle,fill=black!50!black,inner sep=2pt]  at (Z1C) {};
\end{tikzpicture}
\end{center}
 
Note that $\mc{Y}_1^\circ \cap \mc{Z}_1^\circ =\varnothing$ as $\mbf{1}^\intercal\mc{Y}_1\cap \mbf{1}^\intercal\mc{Z}_1 = \set{(3u_3^* - 1)/2}$. We plug in $u_3^*$ check that $\mc{Y}_1, \mc{Z}_1\subset\mc{C}$. By \cref{lem:jm-simplices}, $\left(\U\left(\mc{X}_1\right), \U\left(\mc{Y}_1\right), \U\left(\mc{Z}_1\right)\right)$ is compatible. Finally, we check volumes by \cref{lem:vol-sigma}:
\begin{align*}
\frac{|\mc{X}_1|}{|\mc{Z}_1|} & = \frac{u_3^*-1}{(u_3^*-1)/2 }\left(\frac{\sqrt{2(u_3^*-1)^2}}{\sqrt{2(3-2u_3^*)^2}}\right) = 0.291\dots,\\
\frac{|\mc{X}_1|}{|\mc{Y}_1|} & =\frac{u_3^*-1}{(u_3^*-1)/2}\left(\frac{\sqrt{2(u_3^*-1)^2}}{\sqrt{2(2u_3^*-2)^2}}\right) = 1
\end{align*}
Both ratios are at most $1$, proving \cref{eq:cont-room-c-3}.
\end{proof}

\subsection[Construction for $d=4$]{Construction for \texorpdfstring{$d = 4$}{d = 4}}\label{sec:4-leftover}
Here, $\mc{A}, \mc{C}, \mc{E}\subset [0, 1]^3$ and $u_4^*\approx 1.448$ by \cref{fact:ud-comp}.\begin{proof}[Proof of \cref{prop:cont-leftover} with $d=4$.]
We define the following regions of $\mb{R}^3$ and compute
\begin{align*}
\mc{X}_1 = \mc{A} = \Sigma_2\left(0, (u_4^*-1, 0, 0)\right)
 & \implies \mbf{1}^\intercal \mc{X}_1 = [0, u_4^*-1],\\
\mc{Y}_1 = \Sigma_2\left(u_4^*+\frac{2}{15}, (0, 1, u_4^*-1)\right) & \implies \mbf{1}^\intercal \mc{Y}_1 = \left[u_4^*, u_4^*+\frac{2}{15} \right],\\
\mc{Z}_1 = \Sigma_2\left(u_4^*+\frac{2}{15}, (u_4^*-1, 1, u_4^*-1)\right) & \implies \mbf{1}^\intercal \mc{Z}_1 = \left[u_4^*+\frac{2}{15}, 2u_4^*-1\right], \\
\mc{X}_2 = \Sigma_2\left(\frac{3}{2}, (u_4^*-1, 0, 1)\right)
 & \implies \mbf{1}^\intercal \mc{X}_2 =\left[u_4^*, \frac{3}{2}\right],\\
\mc{Y}_2 = \Sigma_2\left(\frac{3}{2}, (u_4^*-1, 1, 0)\right) & \implies \mbf{1}^\intercal \mc{Y}_2 = \left[u_4^*, \frac{3}{2}\right],\\
\mc{Z}_2 = \mc{E}=\Sigma_2\left(3, (2u_4^*-2, 1, 1)\right) & \implies \mbf{1}^\intercal \mc{Z}_2 = \left[2u_4^*, 3\right].
 \end{align*}
Note that $\mc{Y}_1^\circ \cap (\mc{Z}_1\cup\mc{X}_2\cup\mc{Y}_2)^\circ=\varnothing$ as $\mbf{1}^\intercal\mc{Y}_1\cap (\mbf{1}^\intercal\mc{Z}_1\cup\mbf{1}^\intercal\mc{X}_2\cup\mbf{1}^\intercal\mc{Y}_2)=\set{u_4^*, u_4^* + 2/15}$. We plug in $u_4^*$ to check that $\mc{Y}_1, \mc{Z}_1, \mc{X}_2, \mc{Y}_2\subset\mc{C}$. By \cref{lem:jm-simplices}, $\left(\U\left(\mc{X}_1\right), \U\left(\mc{Y}_1\right), \U\left(\mc{Z}_1\right)\right)$ and $\left(\U\left(\mc{X}_2\right), \U\left(\mc{Y}_2\right), \U\left(\mc{Z}_2\right)\right)$ are both compatible. Finally, we check volumes by \cref{lem:vol-sigma}:
\begin{align*}
\frac{|\mc{X}_1|}{|\mc{Z}_1|} & = \frac{u_4^*-1}{u_4^*-17/15}\left(\frac{\sqrt{2(u_4^*-1)^2}}{\sqrt{2(u_4^*-2)^2}}\right)^2 = 0.937\dots,\\
\frac{|\mc{X}_1|}{|\mc{Y}_1|} & = 
\frac{u_4^*-1}{2/15}\left(\frac{\sqrt{2(u_4^*-1)^2}}{\sqrt{1+(u_4^*-2)^2+(u_4^*-1)^2}}\right)^2 = 0.895\dots
\\ \frac{|\mc{Z}_2|}{|\mc{X}_2|} & = \frac{3-2u_4^*}{3/2-u_4^*}\left(\frac{\sqrt{2(2u_4^*-3)^2}}{\sqrt{1+(u_4^*-2)^2+(u_4^*-1)^2}}\right)^2 = 0.028\dots
\\ \frac{|\mc{Z}_2|}{|\mc{Y}_2|} & = 
\frac{3-2u_4^*}{3/2-u_4^*}\left(\frac{\sqrt{2(2u_4^*-3)^2}}{\sqrt{1+(u_4^*-2)^2+(u_4^*-1)^2}}\right)^2= 0.028\dots
\end{align*}
The first is strictly less than $1$ and the sum of the last three is also less than $1$, proving \cref{eq:cont-room-c}.
\end{proof}

\section{Discretization and the weight Function}\label{sec:discr}

In this section, we discretize the couplings given in \cref{prop:BBD} and \cref{prop:cont-leftover} to construct the weight function, and show that it satisfies the conditions of \cref{prop:weight}.

\subsection{The discretization procedure}

We discretize couplings that witnesses the compatibility of a triple of uniform distributions on convex polytopes with nonempty interiors. This procedure works in arbitrary dimensions. Define the total variation distance between $P$ and $Q$ on the same discrete state space $\Omega$ by 
\begin{equation}
    \dtv(P, Q) = \frac{1}{2}\sum_{\omega \in \Omega}|P(\omega) - Q(\omega)|.
\end{equation}

\begin{lemma}\label{lem:discr-pseudo-uniform}
    For $n\in \mb{N}$ and polytope $\mc{X} \subset [0,1]^{d}$ with non-empty interior, define a random vector $X$ as follows.
    
 \begin{enumerate}
     \item Sample $U_{\mc{X}} \sim \U(\mc{X})$;
     \item Round $nU_{\mc{X}}$ to $X' \in \mb{Z}^d$ such that $\|X' - nU_{\mc{X}}\|_\infty \leq 2$;
     \item Sample $T \sim \U([r]^{d})$ independently with $r = \floor{n^{1/2}}$;
     \item Let $X = X' + T$. 
 \end{enumerate}
 
    Then $\dtv( X, \U(\mc{X}^n)) = O\left(n^{-1/2}\right)$. 

\end{lemma}

\begin{proof}
Let $B(l) \triangleq [-l, l]^d$. Let $\mc{I} = \{x \in \mb{Z}^d : x + B(r) \subset n \mc{X} \}$ be the set of lattice points in $n \mc{X}$ far from the boundary. Roughly speaking, we will show that  $X$ and $\U(\mc{X}^n)$ have similar probabilities inside $\mc{I}$ and tiny probability of falling outside $\mc{I}$.

First, observe that if $x \in \mc{I}$ then\begin{equation*}
\mb{P}(X = x) = \mb{P}(X' + T = x) =\frac{\mb{P}(X' \in x - [r]^{d})}{r^{d}}=\frac{r^d+O(r^{d-1})}{r^d|n\mc{X}|} =\frac{1 + O(n^{-1/2})}{|\mc{X}^n|}.
\end{equation*}by \cref{fact:geometry-of-numbers} and our choice of $r$. On the other hand, $\mb{P}(\U(\mc{X}^n) =x) = 1/|\mc{X}^n|$.

Next, we show that the probability that $X$ or $\U(\mc{X}^n)$ fall outside $\mc I$ is small. Suppose $\partial \mc{X}$ denotes the boundary of $\mc{X}$.
\begin{align*}
	\mb{P}\left(X \notin \mc I \right) &= \mb{P}\left(X + B(r) \not \subset n \mc X \right) \\
		&\leq \mb{P}\left(nU_{\mc X} + 3B(r) \not\subset n\mc{X} \right) \\
		&= \mb{P}\left(U_{\mc X}+ B(3r/n) \not\subset \mc{X} \right) \\ 
        &\leq  \mb{P}\left(U_{\mc X} \in  \mc{\partial X} + B(3r/n) \right)  \\
        &= \frac{\on{vol}_{d-1}(\partial \mc{X}) \times O(r/n) }{\on{vol}_{d}(\mc{X})} =  O\left(n^{-1/2}\right)
\end{align*}

Similarly, we have $\mb{P}(\U(\mc{X}^n) \notin \mc I) = \mb{P}\left(\U(\mc{X}^n)+ B(r) \not\subset n\mc{X} \right) = O\left(n^{-1/2}\right)$. Combining together,
    \begin{align*}
        \dtv( X, \U(\mc{X}^n)) &= \frac{1}{2}\sum_{x \in \mb{Z}^{d}} \left|\mb{P} (X=x) - \mb{P} (\U(\mc{X}^n)=x)\right| \\
                      &\leq \frac{1}{2}\sum_{ x \in \mc I } \left|\mb{P} (X=x) - \mb{P} (\U(\mc{X}^n)=x)\right|  +  \mb{P}\left(X \notin \mc I\right) + \mb{P}\left(\U(\mc{X}^n) \notin \mc I \right)\\
                      &\leq \left| \mc I \right| \cdot \left|\frac{1 + O(n^{-1/2})}{|\mc{X}^n|}- \frac{1}{|\mc{X}^n|}\right| + O\left(n^{-1/2}\right) 
                      \\ & \le O\left(n^d\right) \cdot \frac{O(n^{-1/2})}{\Theta(n^d)}+ O\left(n^{-1/2}\right) \\
                      &= O\left(n^{-1/2}\right).
    \end{align*}
\end{proof}

\begin{lemma}\label{lem:dtv-ineq}
Let $Z, U, T$ be independent random vectors taking values in $\mb{Z}^{d}$. Then,
\begin{equation}
    \dtv(Z + T, U) \leq \dtv(Z,U) + \dtv(U+T, U).
\end{equation}
\end{lemma}

\begin{proof}
By the triangle inequality and data processing inequality, we get
\begin{align*}
	\dtv(Z+T, U) &\leq \dtv (Z+T, U+T) + \dtv(U +T, U) \\
				 &\leq \dtv (Z,U) + \dtv(U+T, U).
\end{align*}
\end{proof}

\begin{theorem}\label{thm:discr}
Let $\mc{X}, \mc{Y}, \mc{Z} \subset [0, 1]^{d}$ be polytopes with non-empty interior.
Suppose the triple of distributions $(\U(\mc{X}), \U(\mc{Y}), \U(\mc{Z}))$ are compatible, then there exists a triple of random vectors $(X, Y, Z) \in \mb{Z}^d \times \mb{Z}^d \times \mb{Z}^d$ for every $n\in\mb{N}$ such that $X + Y = Z$ and
\begin{align}
    \dtv(X, \U(\mc{X}^n)) &= O\left(n^{-1/2}\right), \label{eq:disc-cond-x}\\
    \dtv(Y, \U(\mc{Y}^n)) &= O\left(n^{-1/2}\right),\label{eq:disc-cond-y} \\
    \dtv(Z, \U(\mc{Z}^n)) &= O\left(n^{-1/2}\right).\label{eq:disc-cond-z}
\end{align}

\end{theorem}

\begin{proof}
Let $r = \fl{n^{1/2}}$. We define $(X, Y, Z)$ as follows, mirroring \cref{lem:discr-pseudo-uniform}.
\begin{itemize}
    \item Sample a triple $(X_0, Y_0, Z_0)$ from the coupling witnessing $(\U(\mc{X}), \U(\mc{Y}), \U(\mc{Z}))$.
    \item Pick any $X', Y' \in \mb{Z}^{d}$ such that $\|X' - nX_0\|_\infty \leq 1$ and $ \|Y' - nY_0\|_\infty \leq 1$. 
    \item Set $Z' = X' + Y'\in \mb{Z}^d$.
    \item  Sample $T_1 \sim \U([r]^{d})$ and $T_2 \sim \U([r]^{d})$ independently. 
    \item Set $X = X' + T_1$, $Y = Y' + T_1$, and $Z= Z' + T_1 + T_2$.
\end{itemize}
We claim that the triple $(X,Y,Z)$ satisfy the conditions. 
To begin with, we can apply \cref{lem:discr-pseudo-uniform} to $\mc{X}$, $\mc{Y}$, and $\mc{Z}$, so we get \cref{eq:disc-cond-x,eq:disc-cond-y}. Furthermore, note that $\|Z' - nZ_0\|_\infty \leq 2$ and 
\begin{equation}\label{eq:disc-z-comp}
\dtv(Z' + T_1, \U(\mc{Z}^n)) = O\left(n^{-1/2}\right).
\end{equation} 
By \cref{lem:dtv-ineq} we have that 
\begin{align*}
	 \dtv(Z, \U(\mc{Z}^n)) &= \dtv(Z' + T_1 + T_2, \U(\mc{Z}^n)) \\
	 		&\leq \dtv(Z' + T_1, \U(\mc{Z}^n)) + \dtv(\U(\mc{Z}^n) + T_2, \U(\mc{Z}^n)) \\
	 		&\leq O\left(n^{-1/2}\right) + \dtv(\U(\mc{Z}^n) + T_2, \U(\mc{Z}^n)),
\end{align*}
so it suffices to bound $\dtv(\U(\mc{Z}^n) + T_2, \U(\mc{Z}^n))$.

Let $B(l):= [-l, l]^d$ and $\mc I = \{z \in \mb{Z}^{d} : z + B(r) \subset \mc{Z}^n\}$. Observe that  for each $z \in \mc{I}$ we have $\mb{P}(\U(\mc{Z}^n) + T_2 = x) = \mb{P}(\U(\mc{Z}^n) = x) = 1/|\mc Z^n|$. Therefore,
\begin{equation*}
	\dtv(\U(\mc{Z}^n) + T_2, \U(\mc{Z}^n)) \leq \mb{P}(\U(\mc{Z}^n) + T_2 \notin \mc I) +  \mb{P}(\U(\mc{Z}^n) \notin \mc I) = O(n^{-1/2})
\end{equation*}
This proves \cref{eq:disc-cond-z}.
\end{proof}

\subsection{Construction of the weight function}
Let $d\in\{3, 4\}$. By \cref{prop:BBD} and  \cref{prop:cont-leftover}, the triples $(\U(\mc{X}_i), \U(\mc{Y}_i), \U(\mc{Z}_i))$ are compatible for $i \in \{0,1,2\}$, where 
\begin{gather}
	\mc{X}_0 = \mc{B}, \; \mc{Y}_0 = \mc{B}, \; \mc{Z}_0=\mc{D}\\ 
	\mc{X}_1 = \mc{A}, \; \mc{Y}_1 \subset \mc{C}, \; \mc{Z}_1 \subset \mc{C} \\
	\mc{X}_2 \subset \mc{C}, \; \mc{Y}_2 \subset \mc{C}, \; \mc{Z}_2=\mc{E},
\end{gather}
and they satisfy 
\begin{equation*}
	\frac{|\mc{X}_1|}{|\mc{Y}_1|}+\frac{|\mc{Z}_2|}{|\mc{X}_2|}+\frac{|\mc{Z}_2|}{|\mc{Y}_2|} < 1.
\end{equation*}
We can apply \cref{thm:discr} to the coupling witnessing the compatibility of triple $\left(\U\left(\mc{X}_i\right), \U\left(\mc{Y}_i\right), \U\left(\mc{Z}_i\right)\right)$ to obtain coupling $(X_i, Y_{i}, Z_i) $ on $ \mb{Z}^{d-1} \times \mb{Z}^{d-1} \times \mb{Z}^{d-1}$ such that $X_i + Y_{i} = Z_i$ and
\begin{equation}\label{eq:marginal-tv}
\begin{aligned}
\dtv(X_i, \U(\mc{X}_i^n)) &= O\left(n^{-1/2}\right), \\
\dtv(Y_i, \U(\mc{Y}_i^n)) &= O\left(n^{-1/2}\right), \\
\dtv(Z_i, \U(\mc{Z}_i^n)) &= O\left(n^{-1/2}\right).
\end{aligned}
\end{equation}Let  $w_i:\mb{Z}^{d-1}\times \mb{Z}^{d-1}\times \mb{Z}^{d-1}\to \mb{R}_{\ge 0}$ be defined by  $w_i (x, y, z) = \mb{P}\left((X_i, Y_{i}, Z_i) = (x, y, z)\right)$. Define the weight function
\begin{equation}\label{eq:weight-3omp}
    w(x, y, z) = \left|\mc{Z}_0^n\right|w_0(x, y, z)+\left|\mc{X}_1^n\right|w_1(x, y, z)+\left|\mc{Z}_2^n\right|w_2(x, y, z).
\end{equation}
Under this construction, for any $v\in\mb{Z}^{d-1}$,
\begin{align}
W(v) &= \sum_{x, y \in \mb{Z}^{d-1}} w(v,x,y) + w(x,v,y)+w(x,y,v)\\
     &=  \left|\mc{Z}_0^n\right| \Bigl(\mb{P}(X_0=v) +   \mb{P}(Y_0=v) + \mb{P}(Z_0=v) \Bigr)\\
     	&+\left|\mc{X}_1^n\right|\Bigl( \mb{P}(X_1=v)+ \mb{P}(Y_1=v) + \mb{P}(Z_1=v) \Bigr)\\
     	&+\left|\mc{Z}_2^n\right| \Bigl( \mb{P}(X_2=v) +  \mb{P}(Y_2=v) + \mb{P}(Z_2=v) \Bigr)
\end{align} 
Define  $W'$ and $W''$ by 
\begin{align}
 W'(v) &\triangleq  \left|\mc{Z}_0^n\right| \paren{\frac{\mb{I}\set{v \in \mc{X}_0^n}}{|\mc{X}_0^n|} +  \frac{\mb{I}\set{v \in \mc{Y}_0^n}}{|\mc{Y}_0^n|} + \frac{\mb{I}\set{v \in \mc{Z}_0^n}}{|\mc{Z}_0^n|} }\\
     	&+\left|\mc{X}_1^n\right| \paren{\frac{\mb{I}\set{v \in \mc{X}_1^n}}{|\mc{X}_1^n|} +  \frac{\mb{I}\set{v \in \mc{Y}_1^n}}{|\mc{Y}_1^n|} + \frac{\mb{I}\set{v \in \mc{Z}_1^n}}{|\mc{Z}_1^n|} }\\
     	&+\left|\mc{Z}_2^n\right| \paren{\frac{\mb{I}\set{v \in \mc{X}_2^n}}{|\mc{X}_2^n|} +  \frac{\mb{I}\set{v \in \mc{Y}_2^n}}{|\mc{Y}_2^n|} + \frac{\mb{I}\set{v \in \mc{Z}_2^n}}{|\mc{Z}_2^n|} } \\
     W''(v) &\triangleq  \left|\mc{Z}_0\right| \paren{\frac{\mb{I}\set{v \in \mc{X}_0^n}}{|\mc{X}_0|} +  \frac{\mb{I}\set{v \in \mc{Y}_0^n}}{|\mc{Y}_0|} + \frac{\mb{I}\set{v \in \mc{Z}_0^n}}{|\mc{Z}_0|} }\\
     	&+\left|\mc{X}_1\right| \paren{\frac{\mb{I}\set{v \in \mc{X}_1^n}}{|\mc{X}_1|} +  \frac{\mb{I}\set{v \in \mc{Y}_1^n}}{|\mc{Y}_1|} + \frac{\mb{I}\set{v \in \mc{Z}_1^n}}{|\mc{Z}_1|} }\\
     	&+\left|\mc{Z}_2\right| \paren{\frac{\mb{I}\set{v \in \mc{X}_2^n}}{|\mc{X}_2|} +  \frac{\mb{I}\set{v \in \mc{Y}_2^n}}{|\mc{Y}_2|} + \frac{\mb{I}\set{v \in \mc{Z}_2^n}}{|\mc{Z}_2|} }
\end{align}
Note that \cref{eq:marginal-tv} allows us to approximate $W$ by $W'$ as follows: 
\begin{equation}
	\sum_{v \in \mathbb{Z}^{d-1}} \left|W(v) - W'(v)\right| \leq  \left|\mc{Z}_0^n\right| O\left(n^{-1/2}\right) + \left|\mc{X}_1^n\right|O\left(n^{-1/2}\right) + \left|\mc{Z}_2^n\right|O\left(n^{-1/2}\right) = O\left(n^{d-3/2}\right)
\end{equation}
On the other hand, by \cref{fact:geometry-of-numbers} we can approximate $W$ by $W''$ as
\begin{equation}
	\sum_{v \in \mathbb{Z}^{d-1}}\left|W'(v) - W''(v)\right| \leq O\left(n^{-1}\right)\times O\left(n^{d-1}\right) = O\left(n^{d-2}\right).
\end{equation}
By triangle inequality, we finally get that 
\begin{equation}\label{eq:W-approximation}
	\sum_{v \in \mathbb{Z}^{d-1}} \left|W(v) - W''(v)\right| =  O\left(n^{d-3/2}\right).
\end{equation}
Now observe that by construction, we have  
\begin{equation}
	\frac{|\mc Z_0|}{|\mc X_0|} = \frac{|\mc Z_0|}{|\mc Y_0|} = \frac{|\mc B|}{|\mc D|} = \frac{1}{2} \quad  \text{ and }  \quad \frac{|\mc{X}_1|}{|\mc{Y}_1|}+\frac{|\mc{Z}_2|}{|\mc{X}_2|}+\frac{|\mc{Z}_2|}{|\mc{Y}_2|} < 1.
\end{equation}
We conclude that $W'' : \mb{Z}^{d-1} \to \mb{R}_{+}$  
\begin{equation}\label{eq:W''}
	W''(v) = \begin{cases}
		1 &\text{if } v \in \mc A^n \cup \mc B^n \cup \mc D^n \cup \mc E^n\\
		   < 1 &\text{if } v \in \mc C^n \\
		    0 &\text{otherwise.} 
		    						\end{cases}
\end{equation}

\begin{proof}[Proof of \cref{prop:weight}]
	Since $\Omega \subset [n]^{d-1} \times [n]^{d-1} \times [n]^{d-1}$ we have that 
	\begin{equation}
		\sum_{(x,y,z) \notin  \Omega} w(x,y,z) \leq \sum_{v \notin [n]^{d-1}} W(v) \leq O\left(n^{d-3/2}\right) + \sum_{v \notin [n]^{d-1}} W''(v) = O\left(n^{d-3/2}\right).
	\end{equation}
	This proves \cref{eq:weight-1}. On the other hand, note that by \cref{eq:W-approximation} and \cref{eq:W''}, we have that
	\begin{gather}
		\sum_{v \in \mc A^n \cup \mc B^n \cup \mc D^n \cup \mc E^n} |W(v) - 1| = \sum_{v \in A^n \cup \mc B^n \cup \mc D^n \cup \mc E^n} \left|W(v) - W''(v)\right| = O\left(n^{d-3/2}\right) \\
		\sum_{v \in \mc C^n} |W(v) - 1|_+ = O\left(n^{d-3/2}\right)+ \sum_{v \in \mc C^n} \left|W''(v)-1\right|_+ = O\left(n^{d-3/2}\right).
	\end{gather}
	These prove \cref{eq:weight-2} and \cref{eq:weight-3}.
\end{proof}

\section[Towards a General Dimension $d$]{Towards a General Dimension \texorpdfstring{$d$}{d}}\label{sec:conclusion}
In this last section, we comment on how our approach works for a general dimension $d$. For the $d=2$ case first proved by \cite{ER17}, a slight modification of our approach works. We have included the reproof in \cref{sec:2d}. We discuss what we need to generalize \cref{thm:main} to the following conjecture.
\begin{conjecture}\label{conj:main-all-d}
For all $d\in \mb{N}$, the
size of the largest sum-free subset $S^n_d\subset [n]^d$ is
\[\left|S^n_d\right| = (c_d^* + o(1)) \,n^d \quad 
\text{where}
\quad c_d^* = \max_{u\in [0, d]} \on{vol}_d\left(\{v\in [0, 1]^d: \mbf{1}^\intercal v \in [u, 2u)\}\right).
\]
\end{conjecture}

Our reduction to \cref{prop:weight} in \cref{sec:upper} and our discretization procedure in \cref{sec:discr} hold for any dimension $d \geq 3$. Hence, it suffices to construct the continuous couplings \cref{prop:BBD,prop:cont-leftover}, i.e. we have reduced \cref{prop:weight} to the following two conjectures.
Recall $\mc{P}^d_t = \set{v \in [0,1]^d : 1^\intercal  v = t }$ and $u_d^* =  \argmax_{u} \on{vol}_d\left(\{v\in [0, 1]^d: \mbf{1}^\intercal v \in [u, 2u)\}\right).$ 
\begin{conjecture}\label{conj:key-coup}
$(\on{U}(\mc P^d_{u_d^*}),\on{U}(\mc P^d_{u_d^*}), \on{U}(\mc P^d_{2u_d^*}) )$ are compatible for all $d\in \mb{N}$.
\end{conjecture}

\begin{conjecture}\label{conj:leftover-coup}
For all $d\in\mb{N}$, there exists four simplices  $\mc F_1, \mc F_2, \mc F_3, \mc F_4 \subset \mc C$ such that\begin{itemize}
	\item[-] $(\U(\mc{A}),\U(\mc{F}_1),\U(\mc{F}_2))$ and $\left(\U(\mc{F}_3), \U(\mc{F}_4), \U(\mc{E})\right)$ are compatible;
	\item[-] $|\mc{A}|\sqb{\frac{\mb{I}(v \in \mc F_1)}{|\mc F_1|} + \frac{\mb{I}(v \in \mc F_2)}{|\mc F_2|}} + |\mc{E}|[\frac{\mb{I}(v \in \mc F_3)}{|\mc F_3|} + \frac{\mb{I}(v \in \mc F_4)}{|\mc F_4|}]  < 1$ for every $v\in \mc{C}.$
\end{itemize} 
\end{conjecture}

\appendix

\section{Slice Function}
We discuss how to compute $u_d^*$ numerically. Observe that the function $f:  [0, d] \to \mb{R}$ given by $ f(u) = \vol{ \{v \in [0,1]^d : u \leq \mbf{1}^\intercal  v \leq 2u\}}$ is differentiable with derivative 
\begin{equation}\label{eq:def-f'}
    f'(u) = 2\on{vol}_{d-1}\paren{\{v \in [0,1]^d :\mbf{1}^\intercal v = 2u\}} - \on{vol}_{d-1}\paren{\{v \in [0,1]^d :\mbf{1}^\intercal v = u\}}.
\end{equation}
The derivative $f'$ is a continuous piece-wise polynomial function of degree at most $d-1$. One can compute $u_d^*$ by first order optimality condition $f'(u_d^*) = 0$. The numerical values $u_d^*$ and $c_d^*$ are computed numerically in \cite{ER17} and we record them below. 
\begin{fact}[{\cite[Theorem 6.2]{ER17}}]
\label{fact:ud-comp}
To three decimal places, $u_d^*$ and $c_d^*$ for $d\le 4$ are as follows.
\begin{itemize}
\item[-] $u_1^* = 1/2$ and $c_1^*=1/2$
\item[-] $u_2^* = 4/5$ and $c_2^*=3/5$
\item[-] $u_3^* = (15-\sqrt{15})/10= 1.112\dots$ and $c_3^* = (10+\sqrt{15})/20= 0.693\dots$
\item[-] $u_4^*= 1.448\dots$ and $c_4^* = 0.762\dots$
\end{itemize}
\end{fact}

\section{Joint Mixability of Marginals}
\label{sec:mathematica}
In this section we will prove \cref{lem:gabc-jm}.
\begin{proof}[Proof of \cref{lem:gabc-jm}]
    The proof follows from explicit computation. To describe the distributions in the lemma, recall that Irwin-Hall probability densities $f_2$ and $f_4$ with respect to $\mb{R}$ are defined via 
    \begin{equation*}
        X_1 + X_2 \sim f_2 \text{ and } X_1 + X_2 + X_3 + X_4 \sim f_4.
    \end{equation*}
    Observe that by definition of $\mu_X, \mu_Y,$ and $\mu_Z$, we have  
    \begin{equation*}
        \mu_X(t) =\mu_Y(t) = \frac{f_2(t)f_2(u_4^*-t)}{f_4(u_4^*)} \text{ and }  \mu_Z(t) = \frac{f_2(t)f_2(2u_4^*-t)}{f_4(2u_4^*)}.
    \end{equation*}
    From this, it is clear that $\mu_X = \mu_Y$ are symmetric and unimodal with support $[0, u_4^*]$. Moreover, $\mu_Z$ is symmetric and unimodal with support in $[2u_4^*-2, 2]$. Define $L_X, L_Y$, and $L_Z$  as follows
    \begin{gather*}
        L_X(t) = L_Y(t)= \int_0^{t}\mu_X(x + u_4^*/2) -\mu_X(t + u_4^*/2) \, d x, \\
        L_Z(t) = \int_0^{t}\mu_Z(z + u_4^*) -\mu_Z(t + u_4^*) \, d z. 
    \end{gather*}
    By the third condition of \cref{prop:WW}, we would like to show that for any $K \in (0,1/2)$ the positive numbers $(L_X^{-1}(K),L_Y^{-1}(K), L_{Z}^{-1}(K))$ satisfy triangle inequality. Since $L_X = L_Y$, it suffices to show that $L_Z^{-1}(K) \leq 2L_X^{-1}(K)$. Since $L_Z$ is increasing, it is enough to prove $L_Z(2t) \geq L_X(t)$. This immediately follows if $t \geq 1-u_4^*/2$ as $L_Z(2t) =1/2 \geq L_X(t)$. We will show the other case.

    \medskip
    
    Note that if $x < 1-u_4^*/2$, we also have $x < u_4^*/2$ as $u_4^* > 1$. Hence,
    \begin{equation*}
        \mu_X(x + u_4^*/2) = \frac{f_2(u_4^*/2 + x)f_2(u_4^*/2-x)}{f_4(u_4^*)} = \frac{(u_4^*)^2-4x^2}{4f_4(u_4^*)}.
    \end{equation*}
    Recall that $f_4(u_4^*)=2f_4(2u_4^*)$. For $z \leq u_4^*-1$ we also have $z < 2-u_4^*$ as $2u_4^* < 3$ so
    \begin{equation*}
        \mu_Z(z + u_4^*) = \frac{f_2(u_4^* + z)f_2(u_4^*-z)}{f_4(2u_4^*)} = 2 \left( \frac{ (2-u_4^*)^2-z^2}{f_4(u_4^*)}\right).
    \end{equation*}
    Integrating above, we have that for all $t < (u_4^* - 1)/2$
    \begin{equation*}
        L_Z(2t)  = \frac{32}{3} \left( \frac{t^3}{f_4(u_4^*)}\right) \text{ and } L_X(t) = \frac{2}{3}\left( \frac{t^3}{f_4(u_4^*)}\right).
    \end{equation*}
    This already proves the regime when $t < (u_4^* - 1)/2$. Finally, for $(u_4^*-1)/2 \leq  t < 1-u_4^*/2$ we have
    \begin{equation*}
        L_X(t) \leq L_X(1-u_4^*/2) < L_Z(u_4^*-1) \leq L_Z(2t).
    \end{equation*}
    The second inequality follows from the fact that 
    \begin{equation*}
        \frac{L_Z(u_4^*-1)}{L_X(1-u_4^*/2)} = 2 \paren{\frac{u_4^*-1}{1-u_4^*/2}}^3 \geq 1.
    \end{equation*}
    This finishes the proof. 
\end{proof}

\section{Reproof of the Two Dimensional Case}\label{sec:2d}
We also reprove the result for $d=2$ that $c_2 = 3/5$ with the sharp error bound given by \cite{ER17}.

\begin{theorem}\label{thm:2d-main}
	The size of the largest sum-free set $S_2^{(n)} \subset [n]^2$ satisfies is $3n^2/5 + O\left(n\right)$.
\end{theorem}

The lower bound follows \cref{prop:lower}.
For the upper bound, the proof technique of \cref{thm:main} does not apply immediately: the analog of regions $\mc{B}$ and $\mc{D}$ in \cref{eq:ABCDE} are intervals $[0, 4/5]$ and $[3/5, 1]$. They intersect, so we cannot hope for \cref{prop:weight}. This is unsurprising as the upper bound we seek is $3/5$, but using \cref{lem:nd-lines} cannot give any upper bound smaller than $2/3$.

Instead of lines in the axial direction, we will consider segments in the $\mbf{1}$-direction, and derive an analog of \cref{lem:nd-lines} showing $S^*_2$ is optimal on those segments. Combined with an analog of \cref{prop:weight}, we will derive an stability statement in a region around $S^*_2$. Finally, we will do some numerical computations to rule out cases where $S$ contains any point not in this stability region.
\subsection{Stability of $S_2^*$}
The goal of this section is to deduce the following stability result.
\begin{lemma}\label{lem:2d-stability}
Fix any $\alpha, \beta\geq 0$ such that $3\alpha+\beta \leq 4/5$. Then, $\left|S\right|\leq 3n^2/5+O\left(n\right)$ for any sum-free 
\begin{equation}\label{eq:stab-region}
	S\subset R(\alpha, \beta):=\{v\in [n]^2:\mbf{1}^\intercal v \in [(4/5-\alpha)n, (8/5+\beta)n)\}.
\end{equation}
\end{lemma}

Hence, $S^*_2$ is the largest sum-free subset of $ R(\alpha, \beta)$. Let $[2n]/2 = \{k/2:k\in [2n]\}$ be the set of half-integers. We will prove the lemma by combining triples of segments along $\mbf{1}$ intersecting
\[
\mc{T}:=\{(x, y, z)\in ([2n]/2)^2\times ([2n]/2)^2\times ([2n]/2)^2: x+y=z, \mbf{1}^\intercal x = \mbf{1}^\intercal y =  \lfloor 4n/5 \rfloor\}.
\]
For each triple in $\mc{T}$, the following analog of \cref{lem:nd-lines} holds.

\begin{lemma}\label{lem:2d-lines}
For any $(x, y, z)\in\mc{T}$ and $a, b, c\ge 0$, if $x-a\mbf{1}, y-b\mbf{1}, z+c\mbf{1}\in [0, n]^2$, then 
\begin{equation}\label{eq:in-sq}
\ell\left(x-a\mbf{1}, x+(b+c)\mbf{1}\right), \ell\left(y-b\mbf{1}, y+(a+c)\mbf{1}\right), \ell\left(z-(a+b)\mbf{1}, z+c\mbf{1}\right) \subset [0, n]^2.
\end{equation}
Moreover, for any sum-free set $S\subset [n]^2$,
\begin{equation}\label{eq:2d-seg}
\left|S\cap \ell\left(x-a\mbf{1}, x+(b+c)\mbf{1}\right)\right|+ \left|S\cap \ell\left(y-b\mbf{1}, y+(a+c)\mbf{1}\right)\right|+\left|S\cap \ell\left(z-(a+b)\mbf{1}, z+c\mbf{1}\right)\right|
\end{equation}
is at most $2(a+b+c)+O\left(1\right)$, and equality is attained when $S=S^*_2$.
\end{lemma}
\begin{proof}
Define partial order $\preceq$ on $\mb{R}^2$ where $p\preceq q$ if $p_1 \leq q_1$ and $p_2\leq q_2$. Since $x+y=z$
\[\mbf{0}\preceq x-a\mbf{1}\preceq x+(b+c)\mbf{1} = (z+c\mbf{1})-(y-b\mbf{1})\preceq n\mbf{1}-\mbf{0} = n\mbf{1}.\]
Hence, $x+(b+c)\mbf{1}\in [0, n]^2$, and similarly $y+(a+c)\mbf{1}\in [0, n]^2$ by symmetry. Now
\[ \mbf{0} = \mbf{0}+\mbf{0}\preceq (x-a\mbf{1})+(y-b\mbf{1}) = z-(a+b)\mbf{1} \preceq z\preceq n\mbf{1}\]
so $z-(a+b)\mbf{1}\in [0, n]^2$. Then, \cref{eq:in-sq} follows the convexity of $[0, n]^2$. By an affine transformation, we can apply now apply \cref{lem:nd-lines} to the segments in \cref{eq:in-sq}. The endpoints of the segments being half-integral introduces $O\left(1\right)$ error, so
\[ \cref{eq:2d-seg} \leq 2\left|[n]^2\cap \ell\left(\mbf{0}, (a+b+c)\mbf{1}\right)\right|+O\left(1\right)= 2(a+b+c)+O\left(1\right).\]
When $S=S^*_2$, its boundary intersects the segments in \cref{eq:in-sq} at $x, y, z$ respectively, so
\begin{align*}	
\cref{eq:2d-seg} &  = \left|[n]^2\cap \ell\left(x, x+(b+c)\mbf{1}\right)\right|+ \left|[n]^2\cap \ell\left(y, y+(a+c)\mbf{1}\right)\right|+\left|[n]^2\cap \ell\left(z-(a+b)\mbf{1}, z\right)\right|
\\ &  = \left|[n]^2\cap \ell\left(\mbf{0}, (b+c)\mbf{1}\right)\right|+ \left|[n]^2\cap \ell\left(\mbf{0}, (a+c)\mbf{1}\right)\right|+\left|[n]^2\cap \ell\left(\mbf{0}, (a+b)\mbf{1}\right)\right| +O\left(1\right)
\\ & = 2(a+b+c)+O\left(1\right)\qedhere
\end{align*}

\end{proof}
To combine the inequality on \cref{eq:2d-seg}, we construct a weight function on $\mc{T}$. The idea is similar to \cref{prop:weight}, but do it explicitly in the discrete setting, thereby obtaining a sharper error term.
\begin{lemma}\label{lem:2d-weights}
There exists weight function $w:\mc{T}\to \mb{R}_{\geq 0}$ such that $\sum_{e\in \mc{T}}w(e)=O\left(n\right)$, and the following holds for all but at most $O\left(1\right)$-many $v\in ([2n]/2)^2$ satisfying $\mbf{1}^\intercal v\in \{\lfloor 4n/5 \rfloor, 2\lfloor 4n/5 \rfloor\}$:
\begin{equation}\label{eq:2d-marginal-wt}
    \sum_{e \in \mc{T}:v\in e} w(e)=1.
\end{equation}
\end{lemma}
\begin{proof}
Let $m = \lfloor n / 10\rfloor$ and $r  = \lfloor 4n / 5\rfloor - 8 \lfloor n / 10\rfloor$, so that $8m + r = \lfloor 4n / 5\rfloor$. Define the weight function $w$ to be $1$ on the following triples for all $k \in \{-2m, -2m+1, \cdots, 2m\}$, and $0$ otherwise:
{\small
\begin{gather*}
\left(\left(7m + r + \frac{k}{2}, m - \frac{k}{2}\right), \left(m + \frac{k}{2}, 7m + r - \frac{k}{2}\right), \left(8m + r + \frac{2k}{2}, 8m  + r - \frac{2k}{2}\right)\right), \\
\left(\left(5m + r + \frac{k+1}{2}, 3m - \frac{k+1}{2}\right), \left(3m + \frac{k}{2}, 5m + r - \frac{k}{2}\right), \left(8m + r + \frac{2k+1}{2}, 8m + r - \frac{2k+1}{2}\right)\right).
\end{gather*}
}
Then, for all but at most $O\left(1\right)$-many $v\in ([2n]/2)^2$ with $\mbf{1}^\intercal v\in \{\lfloor 4n/5 \rfloor, 2\lfloor 4n/5 \rfloor\}$, \cref{eq:2d-marginal-wt} holds as only one summand is nonzero, and it is equal to $1$. The total weight is $O\left(m\right)$ which is $O\left(n\right)$.
\end{proof}

We will now combine \cref{lem:2d-lines,lem:2d-weights} to show \cref{lem:2d-stability}.
\begin{proof}[Proof of \cref{lem:2d-stability}]
Similar to before, for all $v\in ([2n]/2)^2$ such that $\mbf{1}^\intercal v = \lfloor 4n/5 \rfloor$, define
\[\lambda_-(v) = \left|S\cap \ell(v-n\alpha\mbf{1}, v+n(\alpha+\beta)\mbf{1})\right| \quad\text{and}\quad
    \lambda_-^*(v) = \left|S^*_2\cap \ell(v-n\alpha\mbf{1}, v+n(\alpha+\beta)\mbf{1})\right|,\]
and  for all $v\in ([2n]/2)^2$ such that $\mbf{1}^\intercal v = 2\lfloor 4n/5 \rfloor$, define 
\[\lambda_+(v) = \left|S\cap\ell(v-2n\alpha\mbf{1}, v+n\beta\mbf{1})\right| \quad\text{and}\quad
    \lambda_+^*(v) = \left|S^*_2\cap \ell(v-2n\alpha\mbf{1}, v+n\beta\mbf{1})\right|.\]
For any $(x, y, z)\in\mc{T}$, let $a, b\in [0, n\alpha]$ and $c\in [0, n\beta]$ each be maximal such that $x-a\mbf{1}, y-b\mbf{1}, z+c\mbf{1}\in [0, n]^2$. They exist since $x, y, z\in [0, n]^2$, so $a=b=c=0$ works. By \cref{lem:2d-lines},
\begin{equation}\label{eq:2d-opt-lines}
    \lambda_-(x)+\lambda_-(y)+\lambda_+(z) \leq \lambda_-^*(x)+\lambda_-^*(y)+\lambda_+^*(z)+O\left(1\right).
\end{equation}

As $3\alpha+\beta \leq 4/5$, the segments in $\lambda_-(v)$ and $\lambda_+(v)$ are pairwise disjoint and their union covers $R(\alpha, \beta)$. For $S\subset R(\alpha, \beta)$, we use \cref{eq:2d-opt-lines} with weights from \cref{lem:2d-weights} to upper bound

{
\setlength{\jot}{10pt}
\begin{align*}
\left|S\right|& = \sum_{{v\in ([2n]/2)^2: \mbf{1}^\intercal v = \lfloor4n/5\rfloor}}\lambda_-(v) + \sum_{{v\in ([2n]/2)^2: \mbf{1}^\intercal v = 2 \lfloor 4n/5 \rfloor}}\lambda_+(v)
\\ & = O\left(n\right)+ \sum_{(x, y, z)\in \mc{T}} w(x, y, z)\left(\lambda_-(x)+\lambda_-(y)+\lambda_+(z)\right)
\\ &\leq  O\left(n\right)+ \sum_{(x, y, z)\in \mc{T}} w(x, y, z)\left(\lambda_-^*(x)+\lambda_-^*(y)+\lambda_+^*(z) +O\left(1\right)\right)
\\ & = O\left(n\right)+\sum_{{v\in ([2n]/2)^2:\mbf{1}^\intercal v = \lfloor 4n/5 \rfloor }}\lambda_-^*(v) + \sum_{{v\in ([2n]/2)^2: \mbf{1}^\intercal v = 2 \lfloor 4n/5 \rfloor }}\lambda_+^*(v)
\\ & = \left|S^*_2\right|+O\left(n\right).\qedhere
\end{align*}
}
\end{proof}

\subsection{Numerical computations}
To prove \cref{thm:2d-main}, we need to rule out instances where some sum-free $S\subset [n]^2$ contains points outside stability region $R(\alpha, \beta)$ for some $\alpha$ and $\beta$. The idea is that having any point at the corners near $\mbf{0}$ or $\mbf{1}$ imposes a lot of conditions, so it cuts out many elements from $S$. More precisely, we have the following two lemmas, corresponding to the two corners.
\begin{lemma}\label{lem:upper}
If sum-free $S\subset [n]^2$ contains some $v$ with $\mbf{1}^\intercal v \geq 17n/10$, then $\left| S\right| \leq 3n^2/5+O\left(n\right)$.
\end{lemma}
\begin{proof}[Proof of \cref{lem:upper}]
We rescale $[n]^2$ to the unit square $[0, 1]^2$. The neglected boundary effect gives $O\left(n\right)$ error. Let $(xn, yn)$ be the point in $S$ that maximizes $r = xy$. If $\mbf{1}^\intercal v \geq 17n/10$ for some $v\in S$, then $r\geq 7/10$. For every $p\in [0, x]\times [0, y]$, $p$ and $v-p$ cannot both be in $S$ as they sum to $v$, so we cut out an area of $r/2$. By maximality of $r$,
\[ \frac{\left|S\right|-O\left(n\right)}{n^2} \leq r+\int_r^1 \frac{r}{t}dt -\frac{r}{2}= \frac{r}{2}-r\log r\leq \frac{7}{20}-\frac{7}{10}\log \left(\frac{7}{10}\right)=0.599\ldots < 3/5,\]
where we note that $r/2-r\log r$ is decreasing on $[7/10, 1]$, so we can bound it by its value at $7/10$.
\end{proof}
\begin{remark}
This proof is a modification of an argument due to Cameron \cite{Cam02}: upon obtaining the upper bound in $r$ in the display, \cite{Cam02} obtained a global upper bound by maximizing it over $r$. The maximizer and maximum both turn out to be $1/\sqrt{e}= 0.606\ldots$. 
\end{remark}

\begin{lemma}\label{lem:lower}
If sum-free $S\subset [n]^2$ contains some $v$ with $\mbf{1}^\intercal v \leq 17n/30$, then $\left| S\right| \leq 3n^2/5+O\left(n\right)$.
\end{lemma}

The proof of \cref{lem:lower} requires significant computation. Before that, we first we show how to combine \cref{prop:lower,lem:2d-stability,lem:upper,lem:lower} to deduce \cref{thm:2d-main}.
\begin{proof}[Proof of \cref{thm:2d-main}]
Recall that \cref{prop:lower} gives the lower bound. For the upper bound, we fix any sum-free set $S\subset [n]^2$ and do case work on possible values of $\mbf{1}^\intercal v$ for $v\in S$.
\begin{itemize}
	\item If some $v\in S$ satisfies $\mbf{1}^\intercal v \geq 17n/10$, \cref{lem:upper} gives the upper bound $\left|S\right|\leq 3n^2/5+O\left(n\right)$.
	\item If some $v\in S$ satisfies $\mbf{1}^\intercal v \leq 17n/30$, \cref{lem:lower} gives the upper bound $\left|S\right|\leq 3n^2/5+O\left(n\right)$.
	\item Otherwise, $S\subset R(\alpha, \beta)$ for $\alpha = 7/30$ and $\beta = 1/10$, as defined in \cref{eq:stab-region}. We can check that $3\alpha+\beta = 4/5$, so \cref{lem:2d-stability} applies to show $\left|S\right| \leq {3}n^2/{5} +O\left(n\right)$.\qedhere
\end{itemize}
\end{proof}
\begin{proof}[Proof of \cref{lem:lower}]
We rescale $[n]^2$ to the unit square $[0, 1]^2$. The neglected boundary effect gives $O\left(n\right)$ error. Let $v=(a, b)$ be the point in $S/n$ with the smallest $L^1$-norm, so $a+b\leq 17/30$. Without loss of generality, assume $a\geq b$. Define the regions
\[ U=\{p\in [0, 1]^2:\mbf{1}^\intercal p\geq 17/10\} \quad \text{and} \quad L = \{p\in [0, 1]^2: \mbf{1}^\intercal p \leq a+b\}.\]
By definition of $v$ and \cref{lem:upper}, we can assume $S/n$ is disjoint from $U$ and $L$. Observe that for any $A, B\subset [0, 1]^2$ such that $B\subset v+A$, each $p\in B$ and $v+p\in A$ cannot both be in $S/n$, so
\begin{equation}\label{eq:pair-inj}
	\left| \frac{S}{n}\cap(A\cup B)\right| \leq \on{Area}(A)\cdot  n^2 +O\left(n\right),
\end{equation}
which we will refer to as ``cutting out region $B$''. For each positive integer $t$, define $L$-shaped regions \[ R_t := \left\{p\in [0, 1]^2: p-tv\not\in [0, 1]^2, p-(t-1)v\in [0, 1]^2\right\}.\]
We see that $R_t - v \subset R_{t-1}$, allowing us to apply \cref{eq:pair-inj}. Moreover, the area of these regions are
\begin{equation}\label{eq:area-Rt}
\on{Area}(R_t) = \begin{cases}
a+b-(2t-1)ab &\text{if }1\leq t<\lceil{1/a}\rceil\\
(1-(t-1)a)(1-(t-1)b) &\text{if }t=\lceil{1/a}\rceil\\
0  &\text{if }t> \lceil{1/a}\rceil
\end{cases}.
\end{equation}

We casework on $\lceil 1/a\rceil$ and cut out an area of at least $2/5$ in all cases, i.e. $(\left|\ol{S}\right|+O\left(n\right))/n^2\geq 2/5$. In the figures below, $U$ and $L$ are colored orange, and we cut them out trivially. The diagonal lines $x+y=t(a+b)$ for positive integer $t$ are colored purple and $x+y=17/10$ is colored green. Using \cref{eq:pair-inj}, we will also cut out blue, red, and pink regions, with the latter two requiring more careful boundary analysis. Grey regions will also require boundary analysis, except we will not cut them out.\\

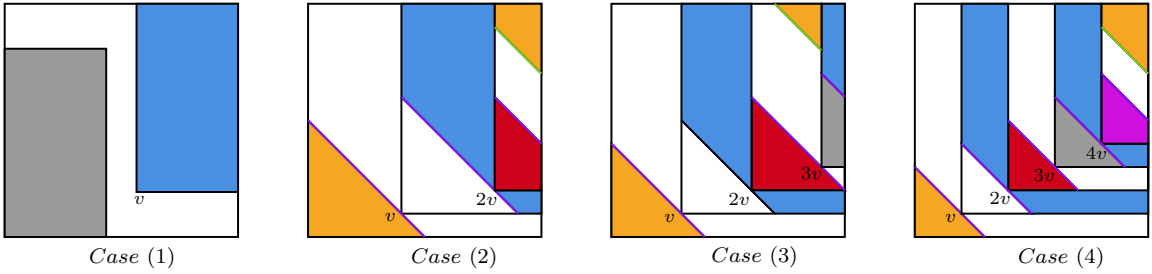
\tikzset{every picture/.style={line width=0.75pt}} 

\begin{tikzpicture}[x=0.75pt,y=0.75pt,yscale=-0.95,xscale=0.95]

\draw   (30,30) -- (153.89,30) -- (153.89,153.89) -- (30,153.89) -- cycle ;
\draw  [draw opacity=0][fill={rgb, 255:red, 74; green, 144; blue, 226 }  ,fill opacity=0.6 ][line width=0.75]  (100,30) -- (153.89,30) -- (153.89,130) -- (100,130) -- cycle ;
\draw   (191.06,30) -- (314.96,30) -- (314.96,153.89) -- (191.06,153.89) -- cycle ;
\draw [color={rgb, 255:red, 0; green, 0; blue, 0 }  ,draw opacity=1 ]   (240.62,141.5) -- (240.62,79.56) ;
\draw  [draw opacity=0][fill={rgb, 255:red, 245; green, 166; blue, 35 }  ,fill opacity=0.6 ] (191.06,91.95) -- (253.01,153.89) -- (191.06,153.89) -- cycle ;
\draw [color={rgb, 255:red, 0; green, 0; blue, 0 }  ,draw opacity=1 ]   (240.62,141.5) -- (302.57,141.5) ;
\draw [color={rgb, 255:red, 144; green, 19; blue, 254 }  ,draw opacity=1 ]   (191.06,91.95) -- (253.01,153.89) ;
\draw  [draw opacity=0][fill={rgb, 255:red, 74; green, 144; blue, 226 }  ,fill opacity=0.6 ] (240.62,30) -- (290.18,30) -- (290.18,129.12) -- (290.18,129.12) -- (240.62,79.56) -- cycle ;
\draw  [draw opacity=0][fill={rgb, 255:red, 74; green, 144; blue, 226 }  ,fill opacity=0.6 ] (290.18,129.12) -- (314.96,129.12) -- (314.96,141.5) -- (302.57,141.5) -- (290.18,129.12) -- cycle ;
\draw [color={rgb, 255:red, 144; green, 19; blue, 254 }  ,draw opacity=1 ]   (240.62,79.56) -- (302.57,141.5) ;
\draw  [draw opacity=0][fill={rgb, 255:red, 245; green, 166; blue, 35 }  ,fill opacity=0.6 ] (290.18,30) -- (314.96,30) -- (314.96,30) -- (314.96,67.17) -- (290.18,42.39) -- cycle ;
\draw [color={rgb, 255:red, 126; green, 211; blue, 33 }  ,draw opacity=1 ]   (290.18,42.39) -- (314.96,67.17) ;
\draw  [draw opacity=0][fill={rgb, 255:red, 208; green, 2; blue, 27 }  ,fill opacity=0.6 ] (290.18,79.56) -- (314.96,104.34) -- (314.96,129.12) -- (314.96,129.12) -- (290.18,129.12) -- cycle ;
\draw [color={rgb, 255:red, 144; green, 19; blue, 254 }  ,draw opacity=1 ]   (290.18,79.56) -- (314.96,104.34) ;
\draw   (352.12,30) -- (476.02,30) -- (476.02,153.89) -- (352.12,153.89) -- cycle ;
\draw [color={rgb, 255:red, 0; green, 0; blue, 0 }  ,draw opacity=1 ]   (389.29,141.5) -- (389.29,91.95) ;
\draw  [draw opacity=0][fill={rgb, 255:red, 245; green, 166; blue, 35 }  ,fill opacity=0.6 ] (352.12,104.34) -- (401.68,153.89) -- (352.12,153.89) -- cycle ;
\draw [color={rgb, 255:red, 0; green, 0; blue, 0 }  ,draw opacity=1 ]   (389.29,141.5) -- (438.85,141.5) ;
\draw [color={rgb, 255:red, 144; green, 19; blue, 254 }  ,draw opacity=1 ]   (352.12,104.34) -- (401.68,153.89) ;
\draw [color={rgb, 255:red, 144; green, 19; blue, 254 }  ,draw opacity=1 ]   (389.29,91.95) -- (438.85,141.5) ;
\draw  [draw opacity=0][fill={rgb, 255:red, 208; green, 2; blue, 27 }  ,fill opacity=0.6 ] (426.46,79.56) -- (476.02,129.12) -- (426.46,129.12) -- cycle ;
\draw  [draw opacity=0][fill={rgb, 255:red, 74; green, 144; blue, 226 }  ,fill opacity=0.6 ] (389.29,30) -- (426.46,30) -- (426.46,129.12) -- (426.46,129.12) -- (389.29,91.95) -- cycle ;
\draw  [draw opacity=0][fill={rgb, 255:red, 74; green, 144; blue, 226 }  ,fill opacity=0.6 ] (426.46,129.12) -- (476.02,129.12) -- (476.02,141.5) -- (438.85,141.5) -- (426.46,129.12) -- cycle ;
\draw  [draw opacity=0][fill={rgb, 255:red, 74; green, 144; blue, 226 }  ,fill opacity=0.6 ] (463.63,30) -- (476.02,30) -- (476.02,129.12) -- (476.02,79.56) -- (463.63,67.17) -- cycle ;
\draw [color={rgb, 255:red, 144; green, 19; blue, 254 }  ,draw opacity=1 ]   (426.46,79.56) -- (476.02,129.12) ;
\draw  [draw opacity=0][fill={rgb, 255:red, 245; green, 166; blue, 35 }  ,fill opacity=0.6 ] (438.85,30) -- (463.63,30) -- (463.63,42.39) -- (463.63,54.78) -- (438.85,30) -- cycle ;
\draw [color={rgb, 255:red, 126; green, 211; blue, 33 }  ,draw opacity=1 ]   (438.85,30) -- (463.63,54.78) ;
\draw  [draw opacity=0][fill={rgb, 255:red, 155; green, 155; blue, 155 }  ,fill opacity=0.6 ] (463.63,67.17) -- (476.02,79.56) -- (476.02,116.73) -- (476.02,116.73) -- (463.63,116.73) -- cycle ;
\draw   (513.19,30) -- (637.08,30) -- (637.08,153.89) -- (513.19,153.89) -- cycle ;
\draw  [draw opacity=0][fill={rgb, 255:red, 245; green, 166; blue, 35 }  ,fill opacity=0.6 ] (513.19,116.73) -- (550.35,153.89) -- (513.19,153.89) -- cycle ;
\draw [color={rgb, 255:red, 144; green, 19; blue, 254 }  ,draw opacity=1 ]   (513.19,116.73) -- (550.35,153.89) ;
\draw  [color={rgb, 255:red, 0; green, 0; blue, 0 }  ,draw opacity=1 ] (537.96,104.34) -- (575.13,141.5) -- (537.96,141.5) -- cycle ;
\draw  [draw opacity=0][fill={rgb, 255:red, 208; green, 2; blue, 27 }  ,fill opacity=0.6 ] (562.74,91.95) -- (599.91,129.12) -- (562.74,129.12) -- cycle ;
\draw  [draw opacity=0][fill={rgb, 255:red, 155; green, 155; blue, 155 }  ,fill opacity=0.7 ] (587.52,79.56) -- (624.69,116.73) -- (587.52,116.73) -- cycle ;
\draw  [draw opacity=0][fill={rgb, 255:red, 74; green, 144; blue, 226 }  ,fill opacity=0.6 ] (537.96,30) -- (562.74,30) -- (562.74,129.12) -- (562.74,129.12) -- (537.96,104.34) -- cycle ;
\draw  [draw opacity=0][fill={rgb, 255:red, 74; green, 144; blue, 226 }  ,fill opacity=0.6 ] (562.74,129.12) -- (637.08,129.12) -- (637.08,141.5) -- (575.13,141.5) -- (562.74,129.12) -- cycle ;
\draw  [draw opacity=0][fill={rgb, 255:red, 74; green, 144; blue, 226 }  ,fill opacity=0.6 ] (587.52,30) -- (612.3,30) -- (612.3,104.34) -- (612.3,104.34) -- (587.52,79.56) -- cycle ;
\draw  [draw opacity=0][fill={rgb, 255:red, 74; green, 144; blue, 226 }  ,fill opacity=0.6 ] (612.3,104.34) -- (637.08,104.34) -- (637.08,116.73) -- (624.69,116.73) -- (612.3,104.34) -- cycle ;
\draw  [draw opacity=0][fill={rgb, 255:red, 207; green, 16; blue, 224 }  ,fill opacity=0.6 ] (612.3,67.17) -- (637.08,91.95) -- (637.08,104.34) -- (637.08,104.34) -- (612.3,104.34) -- cycle ;
\draw  [draw opacity=0][fill={rgb, 255:red, 245; green, 166; blue, 35 }  ,fill opacity=0.6 ] (612.3,30) -- (637.08,30) -- (637.08,129.12) -- (637.08,67.17) -- (612.3,42.39) -- cycle ;
\draw [color={rgb, 255:red, 144; green, 19; blue, 254 }  ,draw opacity=1 ]   (537.96,104.34) -- (575.13,141.5) ;
\draw [color={rgb, 255:red, 144; green, 19; blue, 254 }  ,draw opacity=1 ]   (562.74,91.95) -- (599.91,129.12) ;
\draw [color={rgb, 255:red, 144; green, 19; blue, 254 }  ,draw opacity=1 ]   (587.52,79.56) -- (624.69,116.73) ;
\draw [color={rgb, 255:red, 126; green, 211; blue, 33 }  ,draw opacity=1 ]   (612.3,42.39) -- (637.08,67.17) ;
\draw [color={rgb, 255:red, 144; green, 19; blue, 254 }  ,draw opacity=1 ]   (612.3,67.17) -- (637.08,91.95) ;
\draw [color={rgb, 255:red, 144; green, 19; blue, 254 }  ,draw opacity=1 ]   (463.63,67.17) -- (476.02,79.56) ;
\draw  [draw opacity=0][fill={rgb, 255:red, 155; green, 155; blue, 155 }  ,fill opacity=0.6 ][line width=0.75]  (30,53.89) -- (83.89,53.89) -- (83.89,153.89) -- (30,153.89) -- cycle ;

\draw (97,130.4) node [anchor=north west][inner sep=0.75pt]  [font=\tiny]  {$v$};
\draw (278.62,128.94) node [anchor=north west][inner sep=0.75pt]  [font=\tiny]  {$2v$};
\draw (230.16,139.4) node [anchor=north west][inner sep=0.75pt]  [font=\tiny]  {$v$};
\draw (412.62,129.09) node [anchor=north west][inner sep=0.75pt]  [font=\tiny]  {$2v$};
\draw (378.46,139.98) node [anchor=north west][inner sep=0.75pt]  [font=\tiny]  {$v$};
\draw (451.08,115.44) node [anchor=north west][inner sep=0.75pt]  [font=\tiny]  {$3v$};
\draw (551.19,128.59) node [anchor=north west][inner sep=0.75pt]  [font=\tiny]  {$2v$};
\draw (527.63,139.98) node [anchor=north west][inner sep=0.75pt]  [font=\tiny]  {$v$};
\draw (574.58,116.05) node [anchor=north west][inner sep=0.75pt]  [font=\tiny]  {$3v$};
\draw (602.35,104.66) node [anchor=north west][inner sep=0.75pt]  [font=\tiny]  {$4v$};
\draw (73.19,157.85) node [anchor=north west][inner sep=0.75pt]  [font=\scriptsize]  {$Case\ ( 1)$};
\draw (244.16,157.85) node [anchor=north west][inner sep=0.75pt]  [font=\scriptsize]  {$Case\ ( 2)$};
\draw (403.68,157.29) node [anchor=north west][inner sep=0.75pt]  [font=\scriptsize]  {$Case\ ( 3)$};
\draw (566.4,157.85) node [anchor=north west][inner sep=0.75pt]  [font=\scriptsize]  {$Case\ ( 4)$};

\end{tikzpicture}

\begin{enumerate}
	\item Suppose $a\geq 1/2$. Note that $R_2\subset v+R_1$. By \cref{eq:pair-inj}, we can cut out at least $\on{Area}(R_2)$, so
		\[ \frac{\left|\ol{S}\right|+O\left(n\right)}{n^2}\geq \on{Area}(R_2)=(1-a)(1-b),\]
	which in this case is minimized at $(a,b)=(17/30, 0)$ with value $13/30 > 2/5$.
  
	\item Suppose $1/3\leq a <1/2$. Note that $R_2 \setminus (v+L)\subset v+(R_1\setminus L)$ and $(v+L)\cap [0, 1]^2\subset v+L$, so we can cut out at least $\on{Area}(R_2)$ from the region $R_1\cup R_2$. Also, we see that the red region $(2v+L)\cap [0, 1]^2\subset v+(v+L)\cap [0, 1]^2$, so we also cut it out. It has area
		{\small \[ \frac{1}{2}(a+b)^2 -\frac{1}{2}\max(0,3a+b-1)^2 - \frac{1}{2}\max(0,3b+a-1)^2 \geq \frac{1}{2}(a+b)^2-\frac{1}{2}(3a+b-1)^2-\frac{1}{2}\left(\frac{1}{30}\right)^2,\]}
	where we bound the third term below by checking that $3b+a-1\leq 1/30$, achieved at $(a, b)=(1/3, 7/30)$. We also exclude the orange region $U\cap ([2a,1]\times [2b, 1])$ with area
		{\small \[ \frac{1}{2}\left(2-\frac{17}{10}\right)^2 -\frac{1}{2}\max\left(0,2a-\frac{7}{10}\right)^2 - \frac{1}{2}\max\left(0,2b-\frac{7}{10}\right)^2\geq \frac{9}{200}-\frac{1}{2}\left(2a-\frac{7}{10}\right)^2,\]}
	where we see the third term is $0$ as $b\leq 17/30-a \leq 7/30$. Hence, in total we cut out
	{\small \begin{align*}
			\frac{\left|\ol{S}\right|+O\left(n\right)}{n^2} & \geq \on{Area}(R_2)+\on{Area}\left((2v+L)\cap [0, 1]^2\right)+\on{Area}\left(U\cap ([2a,1]\times [2b, 1])\right)
			\\ &= a+b-3ab + \frac{1}{2}(a+b)^2 - \frac{1}{2}(3a+b-1)^2 - \frac{1}{1800}+ \frac{9}{200} - \frac{1}{2}\left(2a-\frac{7}{10}\right)^2,
		\end{align*}}
	which in this case is minimized at $(a,b)=(1/3, 0)$ with value $0.432\ldots > 2/5$.
	\item Suppose $1/4\leq a <1/3$. As before, we cut out at least $\on{Area}(R_2)$ from the region $R_1\cup R_2$. Now, the red region $(2v+L)\cap [0, 1]^2\subset v+(v+L)\cap [0, 1]^2$ and contains disjoint translations of the grey region $(3v+L)\cap [0, 1]^2$ and trapezoid $L\cap ([1-3a, 1-2a]\times[0, 1])$. We compute
    {\small \begin{equation}\label{eq:slab}
    \begin{aligned}
    \on{Area}\left(L\cap ([1-3a, 1-2a]\times[0, 1])\right) & = \frac{1}{2}\max(0, 4a+b-1)^2-\frac{1}{2}\max(0, 3a+b-1)^2
    \\ & \geq \frac{1}{2}(4a+b-1)^2-\frac{1}{2}(3a+b-1)^2,
    \end{aligned}
    \end{equation}}
    where we use that $a\geq 1/4$. Also, this means the orange region $U\setminus R_4$ has area $(3a-7/10)^2/2$. The blue region $R_4\setminus (3v+L)\subset v+R_3\setminus (2v+L)$, so we also cut it out. In total, we cut out
  {\small 
		\begin{align*}
			\frac{\left|\ol{S}\right|+O\left(n\right)}{n^2} & \geq 
			\on{Area}(R_2) +\on{Area}((2v+L)\cap [0, 1]^2) +\on{Area}\left(R_4\setminus (3v+L)\right)+\on{Area}\left(U\setminus R_4\right)
			\\ &\geq \on{Area}(R_2) +\on{Area}(R_4) + \on{Area}\left(L\cap ([1-3a, 1-2a]\times[0, 1])\right)+ \on{Area}\left(U\setminus R_4\right)
			\\ & = a+b-3ab+(1-3a)(1-3b) + \frac{1}{2}(4a+b-1)^2-\frac{1}{2}(3a+b-1)^2+\frac{1}{2}\left(3a-\frac{7}{10}\right)^2,
		\end{align*}
  }
which in this case is minimized at $a=b= 1/4$ with value $0.4075 > 2/5$.
	\item Suppose $1/5\leq a <1/4$. As before, we cut out at least $\on{Area}(R_2)$ from the region $R_1\cup R_2$. Also, the argument \cref{eq:slab} from case (3) on the grey and red regions still holds, so we can cut out at least $\on{Area}(R_4)+\on{Area}\left(L\cap ([1-3a, 1-2a]\times[0, 1])\right)$ from $R_3\cup R_4$. The pink region $(4v+L)\cap [0, 1]^2\subset v+(3v+L)\cap [0, 1]^2$, so we also cut it out. Now, we do case on $a+b$.
	\begin{enumerate}
	\item Suppose $a+b \geq 17/50$. The top orange region $U\setminus R_5$ overlaps with the pink region $(4v+L)\cap [0, 1]^2$, so we can cut out the entire region $R_5$. Hence, in total, we cut out
		{\small \begin{align*}
			\frac{\left|\ol{S}\right|+O\left(n\right)}{n^2} & \geq 
			\on{Area}(R_2)+\on{Area}(R_4\cup R_5)+ \on{Area}\left(L\cap ([1-3a, 1-2a]\times[0, 1])\right)
    			\\ & \geq a+b-3ab+(1-3a)(1-3b)+\frac{1}{2}(4a+b-1)^2-\frac{1}{2}(3a+b-1)^2,
		\end{align*}}
    which in this case is minimized at $a=b=5/21$ with value $0.404\ldots > 2/5$.
	\item Suppose $a+b < 17/50$. We need to consider the white parallelogram between lines $x+y=5(a+b)$ and $x+y=17/10$, and with $x\in [4a, 1]$. We can compute its area $(4a-1)(5(a+b)-17/10)$ and subtract it from the total we cut out in case (a) to obtain
{\small \[
    \frac{\left|\ol{S}\right|+O\left(n\right)}{n^2} \geq a+b-3ab+(1-3a)(1-3b)+\frac{1}{2}(4a+b-1)^2-\frac{1}{2}(3a+b-1)^2 - (4a-1)\left(5(a+b)-\frac{17}{10}\right),
\]}
    which in this case is minimized at $(a, b)=(1/4, 9/100)$ with value $0.446\ldots > 2/5$.
	\end{enumerate}
	\item Suppose $a<1/5$. As $v\neq\mbf{0}$, there exists a unique positive integer $k\geq 3$ such that either \[ \frac{1}{2k}\leq a< \frac{1}{2k-1}\quad \text{or}\quad \frac{1}{2k+1}\leq a< \frac{1}{2k}.\]
	\begin{enumerate}
	\item For the first case, the non-empty regions are $\{ R_t: t\in [2k]\}$. By \cref{eq:pair-inj}, we include
		\[\frac{\left|S\right|-O\left(n\right)}{n^2}\leq \sum_{t=1}^{k} \on{Area}(R_{2t-1}) = k(a+b-(2k-1)ab),\]
  which under the (slightly relaxed) constraints $1/2k\leq a\leq 1/(2k-1)$ and $0\leq b\leq a$ is maximized at $a=b=1/(2k-1)$ with value $k/(2k-1) \leq 3/5$, as $k\geq 3$.
	\item For the second, the non-empty regions are $\{ R_t: t\in [2k+1]\}$. By \cref{eq:pair-inj}, we cut out
		\[\frac{\left|\ol{S}\right|+O\left(n\right)}{n^2}\geq \sum_{t=1}^{k} \on{Area}(R_{2t}) =k(a+b-(2k+1)ab),\]
which under the (slightly relaxed) constraints $1/(2k+1)\leq a\leq 1/2k$ and $0\leq b\leq a$ is minimized at $a=b=1/2k$ with value $(2k-1)/4k > 2/5$, as $k\geq 3$.
	\end{enumerate} 
\end{enumerate}

In all the cases above, we showed $\left(\left|\ol{S}\right|-O\left(n\right)\right)/n^2\geq 2/5$, so $\left|S\right|\leq 3n^2/5+O\left(n\right)$.
\end{proof}

\bibliographystyle{amsplain0.bst}
\bibliography{main.bib}

\end{document}